\documentclass[twocolumn]{autart}    
\usepackage{mathtools}
\usepackage{nccmath}
\usepackage[many]{tcolorbox}
\usepackage{graphicx}          
\usepackage[utf8]{inputenc}
\usepackage{caption}  
\usepackage{subcaption}
\usepackage{amsmath}    
\usepackage{amssymb}   
\usepackage{amsfonts}  
\usepackage{booktabs}  
\usepackage{xcolor}
\usepackage{enumitem}
\usepackage{bm}

\newtcolorbox{openbox}[1][]{
    breakable,
    enhanced,
    colback=white,
    colframe=black,
    boxrule=0.5pt,
    sharp corners,
    bottomrule at break=0pt,
    toprule at break=0pt,
    boxsep=0pt,       
    top=3pt,          
    bottom=3pt,       
    left=3pt,         
    right=3pt,        
    pad at break=2pt, 
    before skip=3pt,  
    after skip=4pt,   
    #1
}
\newenvironment{proof}{\begin{pf}}{\end{pf}}
\newenvironment{myproof}[1]
  {\par\noindent\textbf{Proof of #1.}\ \ } 
  {\ifhmode\unskip\fi\hspace*{\fill}\qed\par} 
\newcommand{\qedopen}{\hfill $\square$} 
\newcommand{\dd}{\mathop{}\!\mathrm{d}} 
\newtheorem{theorem}{Theorem}[section]
\newtheorem{lemma}{Lemma}[section]
\newtheorem{assumption}{Assumption}[section]
\newtheorem{proposition}{Proposition}[section]

\newtheorem{remark}{Remark}[section]
\newcommand{\rfb}[1]{\mbox{\rm
(\ref{#1})}\ifx\undefined\stillediting\else:\fbox{$#1$}\fi}
\newcommand{\dref}[1]{\mbox{\rm
(\ref{#1})}\ifx\undefined\stillediting\else:\fbox{$#1$}\fi}

\newcommand{\rrn}{{\mathbb R}^n}

\newcommand{\PPP}  {{\mathbf P}}
\newcommand{\n}  {{n_w}}
\newcommand{\rr}  {{\mathbb R}}

\newcommand{\Lscr} {\mathcal{L}}
\newcommand{\Hscr}{\mathcal{H}}
\newcommand{\Ns} {\mathcal{N}}
\newcommand{\col} {{\rm col}}
\newcommand{\diag} {{\rm diag}}

\newcommand{\Mc}{\mathcal{M}}

\renewcommand{\Re}{\mathrm{Re\;}}
\newcommand{\tauy}    {{\tau_y}}

\newcommand{\mt}[1]{\begin{bmatrix}#1\end{bmatrix}} 
\newcommand{\sm}[1]{\left[\begin{smallmatrix}#1\end{smallmatrix}\right]} 
\newcommand{\malign}[1]{
\begin{equation}
\begin{aligned}
#1
\end{aligned}
\end{equation}}

\begin{document}

\begin{frontmatter}

\title{Output Regulation for Linear Parabolic Systems Using Finite-Dimensional Tracking-Error-Based Control \thanksref{footnoteinfo}} 

\thanks[footnoteinfo]{This paper was not presented at any IFAC 
meeting. Corresponding author Bingsen Li.}

\author[Tau,Csu]{Bingsen Li}\ead{bingsenli9@gmail.com},    
\author[Tau]{Emilia Fridman}\ead{emilia@tauex.tau.ac.il},               
\author[Tau]{George Weiss}\ead{gweiss@tauex.tau.ac.il}  

\address[Tau]{School of Electrical \& Computer Engineering, Tel Aviv University, Tel-Aviv, Israel}  
\address[Csu]{School of Mathematics and Statistics, Central South University, Changsha, PR China}             

\begin{keyword}                           
output regulation; parabolic systems; finite-dimensional control; internal model principle, unstable exosystem.          
\end{keyword}                             

\begin{abstract}
This paper addresses the {output regulation problem} for 1-D diffusion-reaction system, where both the disturbance and reference signals are generated by an {unstable exosystem}. We propose a constructive approach to the design of a {finite-dimensional tracking error-based regulator} for this class of possibly unstable systems. Based on the {regulator equations}, a {combined plant} is derived, converting the output regulation problem into a {partial stabilization problem} for the combined system. Using the modal decomposition method, the observability of the truncated modes  is characterized under appropriate  transmission zeros and observability conditions. For the controller design, unlike in the standard stabilization case, where the observer gain is dependent on the unstable modes only, the observer gain here has full order due to the coupling introduced by the exosystem. We prove that the  {observer gain} can be designed such that its norm remains {uniformly bounded} with respect to the dimension of the observer. {LMI-based conditions} are provided for determining the observer dimension, and it is shown that the LMI is feasible for a sufficiently large dimension. Finally, the output regulation problem in the presence of unknown time-varying measurement delays is analyzed. Numerical examples are provided to validate the theoretical results.
\end{abstract}

\end{frontmatter}

\section{Introduction}

The problem of output regulation, which aims to design a controller that ensures that the system output tracks a reference signal and rejects external disturbances, is a fundamental topic in control theory. Since the seminal works \cite{FrancisWonham1976} and \cite{davison1976multivariable}, the internal model principle has become the cornerstone for solving this problem for finite-dimensional linear systems. This principle has subsequently been extended to nonlinear systems,  also with nonlinear exosystems and/or delays; see, for instance, \cite{byrnes1997output}, \cite{fridman2003output}, \cite{knobloch2012topics}, and \cite{Huang2004book}. It was always regarded as desirable to obtain a low (or even minimal) order controller that can solve the output regulation problem, see \cite{natarajan2019minimal}.

Over the past decades, significant efforts have been devoted to extending the output regulation theory to partial differential equation (PDE) systems. The work in \cite{bymes2000output} considered the output regulation problem for PDEs with finite-dimensional exosystems, focusing on plants with bounded control and observation operators. In \cite{natarajan2014state}, the state feedback regulator problem for regular linear systems was solved. Furthermore, \cite{paunonen2017robust} proposed a framework to design robust controllers for error feedback regulation for linear regular systems where the exosystem is infinite-dimensional. The output regulation of impedance passive regular linear systems was investigated under various stability conditions in \cite{Lassi19Passive}. Extensive research has also been conducted for specific PDEs; see, for instance, \cite{DEUTSCHER201556}, 
\cite{guo2020robust}, \cite{LI202501}, and \cite{GuoZhao2022Adptive}.

The above-mentioned works primarily rely on PDE-based observers, which lead to infinite-dimensional controllers. However, practical implementation often necessitates the design of finite-dimensional controllers. Finite-dimensional observer-based controllers for the stabilization of infinite-dimensional systems have been developed using the modal decomposition approach in \cite{RuthCurtain1982}, \cite{Sakawa1983}, \cite{Schumacher1983}, \cite{Deutscher01082010}, \cite{KATZ2020109285}, and \cite{lhachemi2021finite}. Recently, a constructive LMI-based method for finite-dimensional observer-based control of 1-D parabolic PDEs was presented in \cite{KATZ2020109285}. This framework was later extended to systems with unbounded operators and time delays \cite{KATZ2021109364}, \cite{katz2021finite}, \cite{katz2022delayed}. It has been adapted also for basic setpoint control with a known constant reference  \cite{lhachemi2021finite} and with a constant boundary disturbance \cite{lhachemi2026tac}. 

\textbf{In this paper, we propose a constructive approach to design a finite-dimensional tracking error-based regulator for a class of one-dimensional boundary controlled reaction-diffusion equations.} 
As a prerequisite for the controller design, we establish necessary and sufficient conditions for the observability of the truncated combined system (plant+exosystem). Compared with the stabilization problem (e.g., \cite{KATZ2020109285}, \cite{katz2022delayed}, \cite{lhachemi2021finite}), the presence of an exosystem introduces significant challenges, rendering  the previous controller design mechanism and the LMI feasibility analysis methods inapplicable. Consequently, we propose a novel controller design approach—in which the observer gain has full order, with a fresh perspective on feasibility analysis. The specific novelties and challenges of this work are detailed in Remark \ref{Rem:noval}.

Regarding the regulation problem with fi\-nite-di\-men\-sional controllers, \cite{RebarberWeiss2003} proposed a low-gain finite-dimensional controller for exponentially stable well-posed linear systems, and \cite{Xu2019} considered a class of linear time-delay systems, also requiring exponential stability.  In \cite{paunonen2019reduced}, a robust finite-dimensional regulator design was studied for linear systems with analytic semigroups and bounded input/output operators, based on a Galerkin approximation of the original PDEs. However, the approximation dimension must be sufficiently large, and no explicit characterization of the required dimension was provided. Reference \cite{LI2025267} proposed a simple finite-dimensional controller for linear well-posed impedance passive systems. By employing dual observers \cite{Luenberger1971596}, the work in \cite{Deutscher2013} extended the stabilization framework for Riesz-spectral systems \cite{Deutscher01082010} to the output regulation problem, assuming the availability of both the reference signal and an auxiliary output. However, as noted in \cite{Deutscher01082010}, determining the required controller dimension may be difficult to calculate and is highly conservative. Compared with existing finite-dimensional regulator methods such as \cite{Deutscher2013} and \cite{paunonen2019reduced}, our method provides a more efficient and LMI-based quantitative approach to determine the controller order. Furthermore, our approach is capable of handling boundary control and accommodates measurements subject to unknown time-varying delays. 

Other related works include \cite{guo2020robust} and \cite{GuoZhao2022Adptive}, the latter of which was recently extended to include known constant input delays \cite{wang2026output}. These methods rely on PDE-based observers derived from the separation principle, which inherently result in infinite-dimensional controllers. These strategies cannot be adapted to give finite-dimensional observers. This is because, in the regulation problem, the control input must converge to a non-zero steady state; the truncation of a PDE-based observer cannot estimate this desired steady state.

In Section \ref{sec:delay_regulation}, we analyze the output regulation problem with unknown time-varying measurement delays. While output regulation with delays has been investigated in various contexts (e.g., \cite{yoon2016robust}, \cite{fridman2003output}, \cite{lu2015robust}, \cite{LuLiuMutiagent}, \cite{wang2026output}), unknown time-varying measurement delays remain unexplored for PDEs. The ability to handle unknown time-varying measurement delays is another advantage of our method.

\textbf{Notation:}
$\mathcal{H}^n(0,1)$ is the Sobolev space of functions having square-integrable derivatives up to order $n$.  The Euclidean norm on $\mathbb{R}^n$ and the induced matrix norm are denoted by $\|\cdot\|$. $\mathbb{C}_+$ denotes the closed right half-plane, and $\Re \lambda$ is the real part of $\lambda$. For a matrix $P$, $P>0$ means $P$ is symmetric positive definite, and $*$ within a matrix denotes symmetric entries. For a linear operator $A$, $\rho(A)$ and $\sigma(A)$ denote the resolvent set and spectrum, respectively. Let $A$ generate a $C_0$-semigroup on a Hilbert space $Z$. The extrapolation space $Z_{-1}$ is the completion of $Z$ with respect to the norm $\|z\|_{-1} := \|(\beta I - A)^{-1} z\|_Z$ for some $\beta \in \rho(A)$.  $\operatorname{col}\{c_n\}_{n=i}^{j}$ denotes $\sm{c_i\\ \vdots \\ c_{j}}$, and $\operatorname{diag} \{a_n\}_{n=i}^{j}$ denotes the diagonal matrix with entries $a_i, \dots, a_j$. For $0<S\in \rr^{\n\times \n}$ and $x\in \rrn $ we denote $\|x\|_S=\sqrt{x^\top S x}.$ $L(X,Y)$ denotes the space of bounded operators from $X$ to $Y$. 
\section{Problem setup}
Consider a one-dimensional reaction-diffusion system defined on the spatial domain $x \in [0,1]$, with the state space $Z := L^2(0,1)$. The dynamics are governed by the following PDE:
\begin{equation}\label{eq:heat2}
\begin{cases}
    \begin{aligned}
        z^o_t(x,t) &= \frac{\partial}{\partial x} \left(p(x) \frac{\partial}{\partial x} z^o(x,t) \right) \\
                   &\quad + q(x)z^o(x,t) + g(x)\bar{d}_1(t),
    \end{aligned} \\
    z^o_x(0,t) = \bar{d}_2(t), \\ 
    z^o_x(1,t) = u(t) + \bar{d}_3(t), 
\end{cases}
\end{equation}
where $t \geqslant 0$, $z^o(\cdot,t) $ denotes the system state, and $u(t) \in \mathbb{R}$ is the control input. The coefficients satisfy $p(\cdot) \in {C}^2([0,1], \mathbb{R})$ and $q(\cdot) \in {C}^1([0,1], \mathbb{R})$, subject to the bounds:
\begin{equation}\label{coeffi:pq}
    0 < p_m \leqslant p(x) \leqslant p_M, \quad  q_m\leqslant q(x) \leqslant q_M.
\end{equation}
The terms $\bar{d}_i(t)$ ($i=1,2,3$) represent external \textit{disturbances}, and the function $g(\cdot) \in L^2(0,1)$ specifies the spatial distribution of the in-domain disturbance.

We assume that the reference signal $r(t)$ and the disturbances $\bar{d}_i(t)$ are generated by a linear finite-dimensional \textit{exosystem} of the form
\begin{equation} \label{eq:exosystem1}
\begin{aligned}
    \dot{w}(t) &= S w(t), \quad w(0) = w_0 \in \mathbb{R}^\n, \\
    r(t) &= -Q w(t), \\ 
    \bar{d}_i(t) &= d_i w(t), \quad i=1,2,3,
\end{aligned}
\end{equation}
where $w(t) \in W := \mathbb{C}^\n$ is the exosystem state. Here, $Q, d_i\in \mathbb{R}^{1 \times \n}$  are unknown. The spectrum $\sigma(S)$ of $S \in \mathbb{R}^{\n \times \n}$ is known and contained in the closed right half-plane ${\mathbb{C}}_+$. In most output regulation problems, $\sigma(S)$  lies on the imaginary axis.

The regulated output $y(t)$ is defined as
\begin{equation} \label{eq:3.2}
    y(t) = \int_0^1 c(x)z^o(x,t) \dd x, \quad y \in Y := \mathbb{R},
\end{equation}
where $c(\cdot) \in L^2(0,1)$ is a weighting function.

We assume that the \textit{measurement} available to the controller is the \textit{tracking error} $e(t)$, defined as
\begin{equation}\label{tracking error}
    e(t) = y(t) - r(t).
\end{equation}
The objective of the output regulation problem is to design a controller such that $e(t)$ converges to zero asymptotically
 while ensuring the exponential stability of the closed-loop system in the absence of disturbances. 
 
Consider the Sturm-Liouville operator $-A$ defined by:
\begin{equation}\label{operator:SturmL}
\begin{split}
    (-A f)(x) = & -\frac{\dd}{\dd x} \left(p(x)\frac{\dd}{\dd x} f(x)\right) - q(x) f(x), \\
    \mathcal{D}(A) = & \left\{ f \in H^2(0,1) \mid f'(0) = f'(1) = 0 \right\},
\end{split}
\end{equation}
where the coefficients $p(x)$ and $q(x)$ satisfy the conditions given in \eqref{coeffi:pq}. 

The system \dref{eq:heat2}, \dref{eq:exosystem1} and \eqref{tracking error} can be reformulated in the abstract form:
 \begin{subequations}\label{eq:z_simp}
    \begin{align}
        &\dot{f}(t)=A f(t)+Bu(t)+P w(t), \\
        &\dot{w}(t)=Sw(t), \\
        &e(t)=C f(t)+Qw(t),
    \end{align}
\end{subequations}
where  $f(t)=z^o(\cdot,t)$, 
$B\in {L}(U, Z_{-1}), C\in {L}(Z, Y), P\in L(W, Z_{-1}),$ and $Q\in L(W, Y)$,
\[
 B=\delta_1,\ Cf=\langle c, f\rangle_Z , \ P= d_1(\cdot)-d_2 \delta_0+d_3 \delta_1. 
\]
Here $\delta_i, i=0,1,$ denotes the Dirac delta functional, i.e., $\delta_i(\phi)=\phi(i)$ for $\phi\in \mathcal{D}(A^*)$.  

\section{Controller design}
\subsection{Steady state formulation}\label{sec:ErrorSys}
To achieve asymptotic tracking, we seek a steady-state manifold parametrized by the mapping $z_{ss}^o(x,t) = \Pi(x)w(t)$ and a feedforward control $u_{ss}(t) = \Gamma w(t)$ that maintains a zero tracking error $e(t)=0$, while $z_{ss}^o$ and $u_{ss}$ satisfy \eqref{eq:heat2} and $w, r, \bar{d}_i$ satisfy \eqref{eq:exosystem1}. Here, $\Pi(\cdot) \in \Hscr^2(0,1; \mathbb{R}^{1 \times n_w})$ and $\Gamma \in \mathbb{R}^{1 \times n_w}$.

By substituting the steady-state ansatz $z^o(x,t) = \Pi(x)w(t)$ and $u(t)=\Gamma w(t)$ into the system dynamics \eqref{eq:heat2} and utilizing the exosystem dynamics \eqref{eq:exosystem1}, we obtain the boundary value problem
\begin{subequations}\label{eq:regulator1}
    \begin{align}
         & \frac{\dd}{\dd x}\left(p(x)\frac{\dd}{\dd x}\Pi(x)\right) +\Pi(x)(q(x)I - S) = -g(x)d_1, \\
         & \Pi'(0) = d_2, \qquad \Pi'(1) = \Gamma + d_3.
    \end{align}
\end{subequations}
 Furthermore, the zero-error constraint $y(t) - r(t) = 0$ on this manifold requires:
\begin{equation}\label{eq:regulator2}
    \int_0^1 c(x)\Pi(x)\dd x + Q = 0.
\end{equation}
Equations \eqref{eq:regulator1} and \eqref{eq:regulator2} constitute the \textit{regulator equations}.

Applying the Laplace transform to \eqref{eq:heat2}, by linearity, the transfer function of the open-loop system from the boundary input $u$ to the output $y$ is derived as:
\begin{equation} \label{eq:transfer_func}
    \mathbf{P}(s) = \frac{\mathbf{N}(s)}{\mathbf{D}(s)} = \frac{\int_0^1 c(x)\phi(x,s)\dd x}{\phi_x(1,s)},
\end{equation}
where $\phi(x, s)$ is the fundamental solution to the initial value problem (IVP):
\begin{equation} \label{eq:scalar_phi}
\begin{split}
    & -\frac{\dd}{\dd x}\left(p(x)\frac{\dd}{\dd x}\phi(x,s)\right) + (s-q(x))\phi(x,s) = 0, \\
    &  \phi(0,s)=1, \quad \phi_x(0,s)=0.
\end{split}
\end{equation}
Another representation is $\mathbf{P}(s) = \int_0^1 c(x)\psi(x,s) \dd x$, where $\psi(x,s)$ satisfies  \eqref{eq:scalar_phi} with boundary conditions $\psi_x(0,s)=0$ and $\psi_x(1,s)=1$, different from \eqref{eq:scalar_phi}. By linearity and homogeneity, both formulations are equivalent. The IVP formulation \eqref{eq:transfer_func}--\eqref{eq:scalar_phi} is superior for the solvability analysis of the regulator equations.  

To ensure the solvability of the regulation problem, we impose the following assumption, commonly referred to as the transmission zero condition:
\begin{assumption} \label{assump:1}
     The  transfer function $\mathbf{P}(s)$ in \eqref{eq:transfer_func} satisfies $\mathbf{N}(s)\neq 0$ for all $s \in \sigma(S)$.
\end{assumption}

\begin{proposition}\label{propo:solvabilityOfRegu}
The regulator equations \eqref{eq:regulator1}--\eqref{eq:regulator2} admit a unique solution pair $(\Pi, \Gamma)$ if and only if $\mathbf{N}(s)\neq 0$ for all $s \in \sigma(S)$.
\end{proposition}

The proof of Proposition \ref{propo:solvabilityOfRegu} is in the Appendix. 

%
%

Let us define the \textit{transient state} as 
\begin{equation}
    z(x,t) = z^o(x,t) - \Pi(x)w(t),
\end{equation}
 where $\Pi(x)$ is the solution to the regulator equations \eqref{eq:regulator1}. By substituting this into \eqref{eq:heat2} and utilizing the properties of the regulator equations, the dynamics of the transient state $z$ are governed by:
\begin{equation} \label{eq:heat}
\begin{aligned}
     & z_t(x,t) = \frac{\partial}{\partial x} \left(p(x) \frac{\partial}{\partial x} z(x,t) \right) + q(x)z(x,t), \\
     & z_x(0,t) = 0, \quad z_x(1,t) = u(t) - \Gamma w(t), \\
     & \dot{w}(t) = S w(t), \\
     & e(t) = \int_0^1 c(x)z(x,t) \dd x.
\end{aligned}
\end{equation}
Consider the combined system \eqref{eq:z_simp} formed by the plant and the exosystem  on $X:=Z\times W $: 
 \begin{equation}
 	\frac{\dd}{\dd t} \mt{ f(t) \\ w(t)}= \mathcal{A}^p \mt{ f(t) \\ w(t)} +B^pu(t),  \ \ e(t)=C^p  \mt{ f(t) \\ w(t)},
 \end{equation}
 where $
B^p=\Bigl(\begin{smallmatrix}B\\[2pt]0\end{smallmatrix}\Bigr),\
C^p=(C\ \; Q),
$
\[
\mathcal{A}^p=\biggl[
\begin{array}{cc}
A & P \\
0 & S
\end{array}
\biggr],\quad D(\mathcal{A}^p)= \{\sm{f\\ w} \in X| Af+ Pw \in Z\}.
 \] 
We know that \(\mathcal{A}^p\) generates an operator semigroup \(T^p(t)\) on $X$, see \cite[Lemma III.2]{natarajan2014state}.

\begin{assumption}\label{assump:ApCp}
 	The pair $\left( \mathcal{A}^p, C^p \right)$ is approximately observable in infinite time (as defined in \cite{tucsnak2009observation}).
\end{assumption}

This assumption is strictly stronger than assuming that the pair $(A,C)$ in \eqref{eq:z_simp} is approximately observable in infinite time, but it leads to no loss of generality. 
The justification for this was explained in the finite-dimensional setting in~\cite{FrancisBA1977Sicon}; see also~\cite[Proposition~1.4.1]{knobloch2012topics}.  
The following proposition extends this result to the infinite-dimensional case, showing that the unobservability of $\left(\mathcal{A}^p , C^p \right)$ merely indicates a redundancy in the exosystem.  
By eliminating the redundant exosystem variables, one can obtain an equivalent reduced combined system that is observable. 

\begin{proposition}\label{prop:operator-version}
Assume \((A,C)\) is approximately observable in infinite time but \((\mathcal{A}^p, C^p)\) is not. Then
 there exists a bounded invertible operator
$
\mathcal{R}$ on $Z\times W
$
such that, in the new coordinates \(\widetilde x^p=\mathcal{R} x^p, x^p\in X  \), 
\begin{equation}
\begin{aligned}
    \widetilde {\mathcal{A}}^p &= \mathcal{R} \mathcal{A}^p \mathcal{R}^{-1} = \Biggl[
        \begin{array}{cc}
        A & \widetilde P \\[2pt] 0 & \widetilde S
        \end{array}\Biggr] \\
    \widetilde B^p &= \Bigl[\begin{smallmatrix}B\\[2pt] 0\end{smallmatrix}\Bigr], \quad \widetilde C^p = [\,C\ \; \widetilde Q\,].
\end{aligned}
\end{equation}
where, after a suitable decomposition \(W=W_1\oplus W_2\),
\[
\widetilde S =
\biggl[
\begin{array}{cc}
S_{11} & 0 \\
S_{21} & S_{22}
\end{array}
\biggr],\qquad
\widetilde P=(\tilde{P}_1\ \; 0),\qquad
\widetilde Q=(Q_1\ \; 0).
\]
so that the $W_2$ component of $w$ has no influence on the other components of the state or on the output $e$.
Moreover, the reduced subsystem
\begin{equation}\label{pair:reduced}
	\biggl(\biggl[\begin{array}{cc}A & P_1\\[2pt]0 & S_{11}\end{array}\biggr],\; [\,C\ \ Q_1\,] \biggr)
\end{equation}
is approximately observable in infinite time.
\end{proposition} 
 The proof of Proposition \ref{prop:operator-version} is in the Appendix. 

\begin{proposition}\label{propo:SGobservability}
 Under Assumption \ref{assump:ApCp},  let there exist bounded operators $\Pi\in L(W,Z)$ and $\Gamma\in L(W,U) $ such that
\begin{equation}
\label{eq:regulator} 
\Pi S = A \Pi + B \Gamma + P, 
\qquad 0 = C \Pi + Q.
\end{equation}
 Then the pair $(S, \Gamma)$ is observable.
\end{proposition}

The proof is in the Appendix. The observability of the pair $(S,\Gamma)$ plays a crucial role in our finite-dimensional controller design (see Lemma \ref{lem:observability}, Sec. \ref{sec:obser}). As demonstrated there, this property is generically satisfied.

\subsection{Modal decomposition}

It is known that the eigenfunctions of the Sturm-Liouville operator $-A$ in \eqref{operator:SturmL} form an orthogonal basis in $L^2(0,1)$ (see, for instance, \cite[Theorem 1.3.2]{marchenko2011sturm} and \cite[Section 8.2]{egorov1998foundations}). 
 Moreover, these eigenvalues are distinct and can be ordered as $\lambda_n < \lambda_{n+1}$, satisfying the bounds:
\begin{equation}\label{growthCondition}
    \pi^2(n-1)^2 p_m - q_M \leqslant \lambda_n \leqslant \pi^2 n^2 p_M - q_m, 
    \end{equation}
  where we have used notation from \eqref{coeffi:pq}.
The corresponding normalized eigenfunctions $\phi_n(x)$ are uniformly bounded with respect to both the spatial variable $x \in [0, 1]$ and the mode index $n$ (see \cite{Orlov17}).

Consequently, the solution to \eqref{eq:heat} admits the following spectral decomposition:
\begin{equation} \label{eq:3.3}
    z(x,t) = \sum_{n=1}^\infty z_n(t)\phi_n(x), \quad z_n(t) = \langle z(\cdot,t), \phi_n \rangle,
\end{equation}
where $z_n(t)$ represents the modal coefficient corresponding to the eigenvalue $\lambda_n$.
Differentiating \( z_n(t) \) and using integration by parts yields:
\begin{equation}
  \begin{aligned} \label{eq:3.4}
  	& \dot{z}_n(t) = -\lambda_n z_n(t) + b_n (u(t)-\Gamma w(t)), \quad n\geqslant 1, 
  	\\ 
  	& \dot{w}(t)=S w(t),     
  \end{aligned}  
\end{equation}
with the tracking error $e(t)=\sum_{n=1}^{\infty} c_n z_n(t)$.
Here $b_n = p(1) \phi_n(1)$, $c_n=\int_0^1 c(x)\phi_n(x) \dd x$, and $ \lambda_n$ is as in \eqref{growthCondition}. Clearly, $ \{c_n\}\in  \ell^2$ and $b_n$ is uniformly bounded, i.e., there exist $b_M$  such that $|b_n|\leqslant b_M $.  Based on \eqref{eq:heat2}, \eqref{eq:heat} and \eqref{eq:3.4}, the transfer function in \eqref{eq:transfer_func} can be represented as $\PPP(s)=\sum_{i=1}^\infty \frac{b_ic_i}{s+\lambda_i}$.

   We assume that $c_n\neq 0$ for $n\geqslant 1$. Moreover, it can be checked that   $b_n\neq 0$ for $n\geqslant 1$. 
	It is well known that the system operator $A$ in \eqref{eq:z_simp} generates a diagonal semigroup and $(sI-A)^{-1}$ is compact. Thus, $c_n\neq 0$ is equivalent to the fact that the pair $(A,C)$ is approximately observable in infinite time, see \cite[Section 6.9]{tucsnak2009observation}.

Given a desirable decay rate $\delta>0$, define $N_0$ such that  $\lambda_{N_0+1}>\delta$.
Let \( N\geqslant N_0+1 \) be the dimension of the finite dimensional controller that we will apply to the system \eqref{eq:heat2}.  
Denote 
\begin{equation}\label{eq:notations}
\begin{aligned}
    &z^N(t) = \operatorname{col}\{z_i(t)\}_{i=1}^{N}, \quad z^{N_0}(t) = \operatorname{col}\{z_i(t)\}_{i=1}^{N_0}, \\
    &z^{N-N_0}(t) = \operatorname{col}\{z_i(t)\}_{i=N_0+1}^{N}, \quad \hat{z}^N = \operatorname{col}\{\hat{z}_i(t)\}_{i=1}^N, \\
    &A^N = \operatorname{diag}\{-\lambda_n\}_{n=1}^N, \quad A^{N_0} = \operatorname{diag}\{-\lambda_n\}_{n=1}^{N_0}, \\
    &A^{N-N_0} = \operatorname{diag}\{-\lambda_n\}_{n=N_0+1}^{N}, \quad C^N = [c_1, \dots, c_N], \\
    &B^N = \operatorname{col}\{b_i\}_{i=1}^{N}, \quad B^{N_0} = \operatorname{col}\{b_i\}_{i=1}^{N_0}, \\
    &B^{N-N_0} = \operatorname{col}\{b_i\}_{i=N_0+1}^{N}.
\end{aligned}
\end{equation}
The truncated version of system \dref{eq:3.4} is rewritten as 
\begin{equation}\label{eq:ABGaS}
\resizebox{0.99\linewidth}{!}{$
\displaystyle 
   \frac{\dd}{\dd t} \mt{w(t) \\ z^{N_0}(t) \\ z^{N-N_0}(t) } = 
   \mt{S & 0 & 0 \\ -B^{N_0} \Gamma & A^{N_0} & 0 \\ -B^{N-N_0} \Gamma & 0 & A^{N-N_0} }
   \mt{w(t) \\ z^{N_0}(t) \\ z^{N-N_0}(t) } + 
   \mt{0 \\ B^{N_0} \\ B^{N-N_0} } u(t)
$}\end{equation}

The other modal coefficients satisfy,  for all $n>N$, 
   \malign{\label{eq:tail}
  & \dot{z}_n(t) = -\lambda_n z_n(t) + b_n (u(t)-\Gamma w(t)).}

\subsection{Observability of the truncated system.}\label{sec:obser}
 
  As we shall design a finite dimensional observer-based controller, now consider the observability of the following pair
 \begin{equation}\label{pair:ABGaS}
    (\bar{A}, \bar{C})=\left(\mt{S & 0 \\ -B^N \Gamma & A^N}, \mt{0 & C^{N}} \right).
 \end{equation}
 
\begin{assumption}\label{assum:PN}
    	System $(A^N, B^N, C^N)$ has no transmission zeros at $s\in \sigma(S)\bigcap \rho(A^N)$, i.e.,
        \begin{equation}\label{transf:GN}
         G^N(s)= \sum_{i=1}^{N} \frac{b_i c_i}{s+\lambda_i} \neq 0 
        \end{equation}
        {for all} $s\in \sigma(S)\bigcap \rho(A^N).$
\end{assumption} 
Under Assumption \ref{assump:1}, this assumption always holds for sufficiently large  $N$  since $\sum_{i=1}^{N} \frac{b_i c_i}{s+\lambda_i}$ converges pointwise to $\PPP(s)$ for all $s\in \sigma(s)\bigcap \rho(A)$.

\begin{lemma}\label{lem:observability}
The pair in \eqref{pair:ABGaS} is observable {if and only if} the following two conditions hold:
\begin{enumerate}    
    \item Assumption \ref{assum:PN} holds. 
    \item The pair $(S, \Gamma)$ is observable.
\end{enumerate}
\end{lemma}
\begin{proof}
By the Popov-Belevitch-Hautus (PBH) test, the pair $(\bar{A}, \bar{C})$ is observable if and only if the matrix $[\bar{A}^T - \lambda I, \bar{C}^T]^T$ has full column rank for all $\lambda \in \mathbb{C}$. This is equivalent to showing that the only vector  $ v = \col\{\psi, \xi\}$ 
satisfying the following equations is the zero vector:
\begin{align}
    -(\lambda_i + \lambda)\xi_i - b_i \Gamma \psi &= 0, \quad i=1,\dots,N, \label{eq:pbh_x} \\
    (S - \lambda I)\psi &= 0, \label{eq:pbh_y} \\
    \sum_{i=1}^N c_i \xi_i &= 0. \label{eq:pbh_out}
\end{align}
\textbf{(Sufficiency)}: Suppose conditions 1 and 2 hold. We show that $\psi=0$ and $\xi=0$.
First, consider $\psi$. If $\lambda \notin \sigma(S)$, $S-\lambda I$ is nonsingular, implying $\psi=0$. For $\lambda \in \sigma(S)$, we claim $\Gamma \psi = 0$.
\begin{itemize}
    \item If $\lambda \notin \sigma(A^N)$, from \eqref{eq:pbh_x} we have $\xi = -(\lambda I - A^N)^{-1}B^N \Gamma \psi$. Substituting this into \eqref{eq:pbh_out} yields $-G^N(\lambda)\Gamma \psi = 0$. Since $G^N(\lambda) \neq 0$ by Condition 1, it follows that $\Gamma \psi = 0$.
    \item If $\lambda \in \sigma(A^N)$, let $\lambda = -\lambda_k$. From \eqref{eq:pbh_x}, $-b_k \Gamma \psi = 0$. Since $b_k \neq 0$, we have $\Gamma \psi = 0$.
\end{itemize}
Thus, the system reduces to $(S-\lambda I)\psi=0$ and $\Gamma \psi = 0$. Since $(S, \Gamma)$ is observable and $\sigma(S) \subset \mathbb{C}_+$, this implies $\psi=0$.
Substituting $\psi=0$ into \eqref{eq:pbh_x} yields $-(\lambda_i +\lambda)\xi_i = 0$. If $\lambda \notin \sigma(A^N)$, clearly $\xi=0$. If $\lambda = -\lambda_k$, then $\xi_i = 0$ for all $i \neq k$. Equation \eqref{eq:pbh_out} then becomes $c_k \xi_k = 0$, which implies $\xi_k = 0$ since $c_k \neq 0$. Hence, $\xi=0$.

\textbf{(Necessity)}: We show that if either condition fails, the pair is unobservable.
\begin{itemize}
    \item Case 1: Suppose Condition 2 fails, i.e., $(S, \Gamma)$ is not observable. There exists $\lambda \in \sigma(S)$ and $\psi \neq 0$ such that $S \psi = \lambda \psi$ and $\Gamma \psi = 0$. Choose $\xi=0$. Then \eqref{eq:pbh_x} holds since $0 - b_i \cdot 0 = 0$, \eqref{eq:pbh_y} holds by definition, and \eqref{eq:pbh_out} holds since $\xi=0$. Thus, $ v = \col\{\psi, 0\}\neq 0$ is an unobservable mode.
    
    \item Case 2: Suppose Condition 1 fails. There exists $\lambda \in \sigma(S) \bigcap \rho(A^N)$ such that $G^N(\lambda) = 0$. Let $\psi \neq 0$ be an eigenvector of $S$ associated with $\lambda$. Since $\lambda \notin \sigma(A^N)$, the matrix $(\lambda I - A^N)$ is invertible. Choose $\xi = -(\lambda I - A^N)^{-1} B^N \Gamma \psi$. 
    Then \eqref{eq:pbh_x} is satisfied by construction. For the output equation \eqref{eq:pbh_out}, we have:
    \[
    C^N \xi = -C^N (\lambda I - A^N)^{-1} B^N \Gamma \psi = -G^N(\lambda) \Gamma \psi.
    \]
    Since $G^N(\lambda) = 0$, it follows that $C^N \xi = 0$. Thus, $v = \col\{\psi, \xi\}$ is an unobservable mode (note that $\psi \neq 0$ implies $v \neq 0$). \qedopen
\end{itemize}
\end{proof}

\begin{remark} 
     Since the pair $(S, \Gamma)$ must be observable, 
     it constitutes a minimal realization that can generate the signal $\Gamma w(t)$ in \eqref{eq:3.4}. Given that the spectrum of $S$ is known, its characteristic polynomial is uniquely determined. Consequently, one can always construct an equivalent and known pair $(S_{new}, \Gamma_{new})$ (e.g., in the observable canonical form) that can generate the same signal $\Gamma w(t)$. Therefore, for notational simplicity and without loss of generality, we assume that $(S, \Gamma)$ is known.
\end{remark}
\subsection{Observer-based controller construction} \label{sec:obser_constr}
Consistent with the problem formulation, we assume that the tracking error $e(t)$ is the only available measurement. Consider the truncated model \eqref{eq:ABGaS}.
Define $N_1=N_0+\n$ and $N_2=N-N_0$. Re-segment the pair \eqref{pair:ABGaS} with respect to the coordinates $\rr^{N_1}\oplus \rr^{N_2}$:
\begin{equation}\label{pair:lemmaGain}
	\bar{A}= \begin{bmatrix} A_S & 0 \\ \mathcal{B} & A^{N-N_0} \end{bmatrix}, \quad 
    \bar{C}= \begin{bmatrix} \mathcal{C}_1 & C^{N-N_0} \end{bmatrix},
\end{equation}
where 
\begin{equation}\label{def:AS_A2_B}
    \begin{aligned}
    & {A}_S = \mt{S & 0 \\ -B^{N_0}\Gamma & A^{N_0}},  \ \mathcal{B}   = \mt{-B^{N-N_0}\Gamma & 0_{N_2\times N_0} }, \ \ 
     \\ 
    & \mathcal{C}_1 = [0_{1\times n_w} \;\; C^{N_0}], \ \ \ C^{N-N_0} = [c_{N_0+1}, \dots, c_N], \ \\
    & C^{N_0} = [c_1, \dots, c_{N_0}].
    \end{aligned}
\end{equation}
Here   $A_S\in \rr^{N_1 \times N_1}$, ${A^{N-N_0}}\in \rr^{N_2\times N_2}$, $\mathcal{B}\in \rr^{N_2\times N_1}$, $B^{N-N_0}\Gamma \in \rr^{N_2\times \n }$, $\mathcal{C}_1\in \rr^{1 \times N_1}$, and $C^{N-N_0}\in \rr^{1 \times N_2}$. 

We construct the finite-dimensional observer for \eqref{eq:heat} as follows:
{ 
\begin{equation} \label{obsererfinite}
\begin{aligned}
    \hat{z}(x,t) &:= \sum_{n=1}^N \hat{z}_n(t)\phi_n(x), \\
    \frac{\dd}{\dd t} \mt{ \hat{w}(t) \\ \hat{z}^N(t) } &= \mt{S & 0 \\ -B^N\Gamma & A^N } \mt{\hat{w}(t) \\ \hat{z}^N(t)} + \mt{0 \\ B^N } u(t) \\
    &\quad + \mt{L^w \\ L^N } \left( \mt{0 & C^N} \mt{\hat{w}(t) \\ \hat{z}^N(t)} - e(t) \right).
\end{aligned}
\end{equation}
}
 where  
 $\sm{L^w \\ L^N} $ is the \textit{observer gain} designed by using the following observer design algorithm(\textbf{ODA}): \\
\begin{openbox}
\textbf{Step 1.} 
Find the solution $\mathcal{K}\in \mathbb{R}^{N_2\times N_1}$ to the Sylvester equation
\begin{equation}\label{eq:sylvester}
	{A^{N-N_0}} \mathcal{K} - \mathcal{K} A_S = -\mathcal{B}.
\end{equation}
\textbf{Step 2.}
Since  $(A_S, (\mathcal{C}_1+C^{N-N_0}\mathcal{K}))$ is observable (see Lemma \ref{lemma:invarience-gain} below), we
choose $\tilde{\mathbb{L}}_1$ by employing Ackermann's formula (e.g., \cite[Chapter 10.2]{ogata2010modern}) such that  $A_S +\tilde{\mathbb{L}}_1 (\mathcal{C}_1 + C^{N-N_0} \mathcal{K})+\delta I \in \mathbb{R}^{N_1\times N_1}$ is Hurwitz. \\
\textbf{Step 3.} Obtain the observer gain as
\begin{equation}\label{gain:obsv_ori}
	\col\{L^w , L^N \}=\col \{ \tilde{\mathbb{L}}_1, \mathcal{K} \tilde{\mathbb{L}}_1\},
	\end{equation}
    where $L^w\in \mathbb{R}^{n_w}, L^N\in\mathbb{R}^{N},$ and $\tilde{\mathbb{L}}_1\in \mathbb{R}^{N_1}. $
\end{openbox}

The existence of a unique solution to \eqref{eq:sylvester} is guaranteed by $\sigma(A_S) \cap \sigma(A^{N-N_0}) = \emptyset$ (see, for instance, \cite{Sylvester1991}). The motivation behind this design is to keep the observer gain norm uniformly bounded in $N$. 

We choose the \textit{control law} 
\begin{equation}\label{controlLaw}
    u(t)=\mt{\Gamma & K_0} \sm{\hat{w}(t) \\ \hat{z}^{N_0}(t)},
\end{equation}
where $K_0$ is chosen such that, for some $0<P_{K_0} \in \rr^{N_0\times N_0} $,
\begin{equation*}
	(A^{N_0} + B^{N_0}K_0)^\top P_{K_0}+P_{K_0}(A^{N_0} + B^{N_0}K_0 )<-2\delta P_{K_0}.
\end{equation*}
Define the estimation errors as $\tilde{z}_n(t)=\hat{z}_n(t)-z_n(t)$ and $\tilde{w}(t)=\hat{w}(t)-w(t)$. Denote
\begin{equation}\label{eq:def_X0}
\begin{aligned}
    \tilde{z}^N(t) &= \col\{\tilde{z}_1(t), \dots, \tilde{z}_N(t)\}, \\
    X_0(t) &= \col\{ \hat{z}^N(t), \tilde{w}(t), \tilde{z}^N(t)\}. 
\end{aligned}
\end{equation}      
 Based on \eqref{eq:ABGaS}, \eqref{obsererfinite}, \eqref{controlLaw}, and \eqref{eq:def_X0},   the \textit{closed-loop system} is summarized as:
 \begin{equation}\label{closedMod}
 	\begin{aligned}
 		 & \dot{z}_n(t)=-\lambda_n  z_n(t) +b_n \bar{K} X_0(t), \quad n\geqslant N+1, \\  
         & \dot{X}_0(t)=F X_0(t) + \mathcal{L} \zeta (t),      
 	\end{aligned}
 \end{equation}
where $\bar{K}\in \rr^{1\times (2N+\n)}, \mathcal{L}\in \rr^{(2N+\n)\times 1}, A_K\in \rr^{N\times N},$\\$ A_L\in \rr^{(N+\n)\times (N+\n)}$,  and
\begin{equation} \label{eq:system_matrices}
\begin{aligned}
    &\mathcal{L} = \mathrm{col}\{ -L^N , -L^w, -L^N \}, \quad 
    \zeta(t) = \sum_{n=N+1}^{\infty} c_n z_n(t),
     \\
     &F=\left[
    \begin{array}{c|cc}
    A_K
    & 0 & L^N C^N \\
    \hline
    0 & \multicolumn{2}{c}{A_L} \end{array}\right],\ A_L = \mt{S & L^w C^N \\ -B^N \Gamma & A^N + L^N C^N}, \\
    &   A_K=\mt{
A^{N_0}+B^{N_0}K_0 & \ 0 \\[1mm]
B^{N-N_0} K_0 & {A^{N-N_0}} 
},\ \ \bar{K} = \mt{K & \Gamma & 0_{1\times N} }, \\
   &    K=[K_0\ \  0_{1\times N_2}].
   \end{aligned}
\end{equation}

Consider the closed-loop system \eqref{eq:heat}, \eqref{obsererfinite} and \eqref{controlLaw}. Following arguments in \cite{katz2022delayed}, for all initial conditions $z(\cdot, 0) \in \mathcal{H}^1(0,1)$, the closed-loop system has a unique solution satisfying $z \in C([0, \infty); L^2(0,1))\cap C^1((0, \infty); L^2(0,1)) $.

\section{Stability analysis} 

\begin{lemma}\label{propo:Knorm}
Let $A_S, {A^{N-N_0}}$, and $\mathcal{B}$ be defined as in \eqref{def:AS_A2_B}.  
Then, the solution $\mathcal{K}$ of the Sylvester equation
\eqref{eq:sylvester}
satisfies:
$
\| \mathcal{K} \| \leqslant M_0,
$
where $M_0 > 0$ is independent of the dimension $N$.
Moreover, as $N \to \infty$, the sequence of the row norms of $\mathcal{K}$ is in $\ell^1$.
\end{lemma}

\begin{proof}
Since ${A^{N-N_0}}$ is a diagonal matrix, the Sylvester equation \eqref{eq:sylvester} can be decoupled into independent linear equations for the rows of $\mathcal{K}$. Let $k_i^T$ and $v_i^T$ denote the $i$-th rows of $\mathcal{K}$ and $\mathcal{B}$, respectively, for $i = 1, \dots, N-N_0$. Denote $\bar{\lambda}_i=\lambda_{i+N_0}.$ Substituting the diagonal structure of ${A^{N-N_0}}$ into \eqref{eq:sylvester} yields
\[
-\bar{\lambda}_i k_i^T - k_i^T A_S = -v_i^T,
\]
which, upon transposition, becomes
\begin{equation}\label{eq:row_system}
    (A_S^T + \bar{\lambda}_i I_{N_1}) k_i = v_i,
\end{equation}
where $I_{N_1}$ is the identity matrix of dimension $N_1$.

The spectral separation condition
\[
\max\{\lambda:\lambda\in \sigma({A^{N-N_0}}) \} < \min \{\mathrm{Re}\;\lambda:\,\lambda\in \sigma(A_S) \}
\]
 implies that $-\bar{\lambda}_i \notin \sigma(A_S)$ for all $i$, ensuring the invertibility of the matrix coefficient in \eqref{eq:row_system}. Consequently,  $k_i = (A_S^T + \bar{\lambda}_i I_{N_1})^{-1} v_i$. To estimate the norm of $k_i$, we explicitly bound the spectral norm of  $(A_S^T + \bar{\lambda}_i I_{N_1})^{-1}$. Recall that for any nonsingular matrix $M$, $\|M^{-1}\| = 1/\sigma_{\min}(M)$, where $\sigma_{\min}(\cdot)$ denotes the smallest singular value. Applying the inequality $\sigma_{\min}(X + Y) \geqslant \sigma_{\min}(X) - \|Y\|$ with $X = \bar{\lambda}_i I$ and $Y = A_S^T$, we obtain
\[
\sigma_{\min}(A_S^T + \bar{\lambda}_i I_{N_1}) \geqslant \bar{\lambda}_i - \|A_S^T\|.
\]
By the growth condition \eqref{growthCondition}, $\bar{\lambda}_i$ increases quadratically with $i$. Consequently, for sufficiently large $i$, the term $\bar{\lambda}_i$ dominates the constant $\|A_S^T\|$. Specifically, there exists a threshold $i^*$ such that for all $i > i^*$, we have $\bar{\lambda}_i > 2\|A_S^T\|$, implying $\bar{\lambda}_i - \|A_S^T\| > \bar{\lambda}_i / 2$. Therefore,
\begin{equation}\label{eq:resolvent_bound}
    \|(A_S^T + \bar{\lambda}_i I_{N_1})^{-1}\| \leqslant \frac{1}{\bar{\lambda}_i - \|A_S^T\|} \leqslant \frac{2}{\bar{\lambda}_i} \leqslant \frac{C_{\text{res}}}{i^2},
\end{equation}
for some constant $C_{\text{res}} > 0$. 

It is easy to see that there exists a constant $b>0$ such that $|(\mathcal{B})_{ij}| \leqslant b$ for all $i, j$.
Regarding the vector $v_i$, the element-wise bound $|(\mathcal{B})_{ij}| \leqslant b$ implies
\begin{equation}\label{eq:input_bound}
    \|v_i\| = \sqrt{\sum_{j=1}^{N_1} |(\mathcal{B})_{ij}|^2} \leqslant b\sqrt{N_1} \triangleq C_v,
\end{equation}
where $C_v$ is independent of the dimension $N$, as $N\to \infty$.

Combining \eqref{eq:resolvent_bound} and \eqref{eq:input_bound}, the norm of the $i$-th row of $\mathcal{K}$ satisfies
\[
\|k_i\| \leqslant \|(A_S^T + \bar{\lambda}_i I_{N_1})^{-1}\| \|v_i\| \leqslant \frac{C_{\text{res}} C_v}{i^2}, \quad \text{for } i > i^*.
\]
It follows that the sequence of the row norms of $\mathcal{K}$ is in $\ell^1$  as $N \to \infty$.

Finally, we establish the uniform boundedness of $\|\mathcal{K}\|$. 
Note that
$
\|\mathcal{K}\|^2 \leqslant \|\mathcal{K}\|_F^2 = \sum_{i=N_0+1}^{N} \|k_i\|^2.
$
Splitting the summation at $i^*$, the sum of the initial terms of the series is finite. The tail of the series is bounded since the series
$
\sum_{i=i^*+1}^{N} \left( \frac{C_{\text{res}} C_v}{i^2} \right)^2 
$ converges.
 Thus, $\|\mathcal{K}\|$ is bounded uniformly by a constant $M_0$ independent of $N$. 
 \qedopen
\end{proof}
To analyze the stability of the closed-loop system, we will use the Lyapunov method, where the uniform boundedness of the solution of a related Lyapunov equation plays an important role. We need the following lemma: 
\begin{lemma}\label{proposi:F2}
Let
$\bar{F}_\delta = \begin{bmatrix} F_0 & 0 \\ F_2 & {A^{N-N_0}} \end{bmatrix}+\delta I,$
where $F_0 \in \mathbb{R}^{N_1\times N_1}$,
$F_2 \in \mathbb{R}^{N_2 \times N_1 }$, and
${A^{N-N_0}}$ is defined in \eqref{eq:notations}.
Suppose there exist constants $M,\alpha, \beta>0$
such that
$
\|e^{(F_0+\delta I) t}\| \le M e^{-\alpha t}
$ and $|(F_2)_{ij}|\leqslant \beta $ for all entries of $F_2$.
Then there exist $M_1, \delta_1>0$, independent of $N$, such that
$
\|e^{\bar{F}_\delta t}\| \le M_1 e^{-\delta_1 t},
$ for all $t\geqslant 0$.
\end{lemma}
\begin{proof}
Let $ x_0(t)=\col\{x_1(t), \xi(t) \}$ and define $\dot{x}_1(t) = A_1 x_1(t)$ and $\dot{\xi}(t) = F_2 x_1(t) + A_2 {\xi}(t)$, where $A_1 = F_0 + \delta I$ and $A_2 = \text{diag}\{-\mu_n\}_{n=N_0+1}^N$ with $\mu_n = \lambda_n - \delta > 0$. Since $\|e^{A_1 t}\| \le M e^{-\alpha t}$, there exists $P_1>0$, independent of $N$, such that $P_1 A_1 + A_1^T P_1 = -I_{N_1}$.

Consider the Lyapunov function $V(t) = x_1^T(t) P_1 x_1(t) + p_0 \|\xi(t)\|^2$ for some $p_0 > 0$. Differentiating $V$ along the trajectories, and using $P_1 A_1 + A_1^T P_1 = -I_{N_1}$ yields
\begin{equation*}
    \dot{V} = -\|x_1\|^2 - 2p_0 \sum_{n=N_0+1}^N \mu_n \xi_n^2 + 2p_0 \sum_{n=N_0+1}^N \xi_n (F_2)_n x_1,
\end{equation*}
where $(F_2)_n$ is the $n$-th row of $F_2$ satisfying $\|(F_2)_n\|^2\leqslant \bar{\beta}^2 = \beta^2 N_1 $. To estimate the last term, we apply Young's inequality, $2ab \le \mu_n a^2 + b^2/\mu_n$, which bounds the last term by $p_0 \sum_{n>N_0} (\mu_n \xi_n^2 + \bar{\beta}^2 \|x_1\|^2 / \mu_n)$. Then, we have
\begin{equation*}
    \dot{V} \le - \left( 1 - p_0 \sum_{n=N_0+1}^{N} \frac{\bar{\beta}^2}{\mu_n} \right) \|x_1\|^2 - p_0 \mu_{N_0+1} \|\xi\|^2.
\end{equation*}
Since $\lambda_n \sim \mathcal{O}(n^2)$, the series $\sum_{n=N_0+1}^{\infty} \frac{\bar{\beta}^2}{\mu_n} < \infty$ converges independent of $N$. Choosing $p_0>0$ small enough ensures $\dot{V} \le -p_3 \|x_0\|^2$ for some $p_3>0$ independent of $N$. Note that there exist constants $p_1, p_2 > 0$ independent of $N$ such that $p_1 \|x_0\|^2 \le V \le p_2 \|x_0\|^2$. This implies $\dot{V} \le - \frac{p_3}{p_2} V$. The conclusion then directly follows from the comparison principle. \qedopen
\end{proof}
\begin{lemma}\label{lemma:invarience-gain}
Consider the pair \eqref{pair:lemmaGain}. Suppose that Assumptions \ref{assump:1}--\ref{assump:ApCp} hold.
Then  $(A_S, (\mathcal{C}_1+C^{N-N_0}\mathcal{K}))$ is observable, and the gain  in \eqref{gain:obsv_ori}, denoted by $\mathbb{L}=\col \{ \tilde{\mathbb{L}}_1, \mathcal{K} \tilde{\mathbb{L}}_1\}$,  has the following properties:
\begin{enumerate}
    \item $\|\mathbb{L}\|$ is uniformly bounded as $N \to \infty$;
    \item There exist $\delta_1>0$ and $M>0$ (independent of $N$) such that 
     $ \|e^{(\bar{A}+\mathbb{L}\mathcal{\bar{C}}+\delta I)t}\|\leqslant M e^{-\delta_1 t}.    $    
    \end{enumerate}
\end{lemma}
\begin{proof}
 By invoking Proposition \ref{propo:SGobservability} and Lemma \ref{lem:observability}, based on the assumptions of this lemma, the pair \eqref{pair:ABGaS} is observable.  
From Lemma \ref{propo:Knorm}, \eqref{eq:sylvester} has a solution $\mathcal{K}$ with  bounded norm $\|\mathcal{K}\| \leqslant M_1$ independently of $N$ for some $M_1>0$.

Consider the similarity transformation defined by: 
 \[ \mathbb{T} =
\biggl[
\begin{array}{cc}
I & 0 \\[2pt]
\mathcal{K} & I
\end{array}
\biggr],
\quad
\mathbb{T}^{-1} =
\biggl[
\begin{array}{cc}
I & 0 \\[2pt]
-\mathcal{K} & I
\end{array}
\biggr].
\]
Applying it to the system matrices and using \eqref{eq:sylvester} yield:
\begin{equation}
\resizebox{0.92\linewidth}{!}{$
    \tilde{A} = \mathbb{T}^{-1} \bar{A} \mathbb{T} = 
    \begin{bmatrix} A_S & 0 \\ 0 & A^{N-N_0} \end{bmatrix}, \ 
    \tilde{C} = \bar{C} \mathbb{T} = 
    \begin{bmatrix} \mathcal{C}_{N_1} & C^{N-N_0} \end{bmatrix},
$}
\end{equation}
where  $\mathcal{C}_{N_1} = \mathcal{C}_1 + C^{N-N_0} \mathcal{K}\in \mathbb{R}^{1\times N_1}$ depends on $N$ since $C^{N-N_0} \mathcal{K}$ depends on $N$.
Since similarity transformations preserve observability,  the pair $(\tilde{A}, \tilde{C})$ is observable, and so the pair $(A_S, \mathcal{C}_{N_1})$ is observable.

In the transformed coordinates, the observer gain $\tilde{\mathbb{L}} = \begin{bmatrix} \tilde{\mathbb{L}}_1 \\ 0 \end{bmatrix}$ where $\tilde{\mathbb{L}}_1$ is designed as in step 4. The resulting matrix is 
\begin{equation}\label{triangularMa}
	\tilde{A} + \tilde{\mathbb{L}} \tilde{C} = 
\begin{bmatrix} 
A_S + \tilde{\mathbb{L}}_1 \mathcal{C}_{N_1} & \tilde{\mathbb{L}}_1 C^{N-N_0} \\
0 & {A^{N-N_0}} 
\end{bmatrix},
\end{equation} 
which is exponentially stable with decay rate $\delta$.

We now focus on the uniform boundedness of $\tilde{\mathbb{L}}_1$ in $N$.
Since the components of $C^{N-N_0}$ and the row norms of $\mathcal{K}$ both form square-summable sequences (see Lemma \ref{propo:Knorm}), 
 $\mathcal{C}_{N_1}$  converges to a limiting vector $\mathcal{C}_f \in \mathbb{R}^{1\times N_1}$ as $N\to \infty$.
By Proposition \ref{propo:SGobservability}, Assumptions \ref{assump:1} and \ref{assump:ApCp}, the pair $(S, \Gamma)$ is observable.
Based on  Lemma \ref{lem:observability} and \eqref{pair:lemmaGain}, the pair $(\bar{A},{\bar{C}})$  is observable if and only if the  transfer function $G^N(s)$ in \eqref{transf:GN} does not vanish at any $s  \in \sigma(S)\bigcap \rho(A^N)$.
As $N \to \infty$, the transfer function of the truncated system converges pointwise to the infinite-dimensional transfer function $\PPP(s)$ for all $s \in \sigma(S)\bigcap \rho(A)$.
By Assumption \ref{assump:1}, $\PPP(s) \neq 0$ for all $s \in \sigma(S)\bigcap \rho(A)$,
and so $(\bar{A},{\bar{C}})$ is observable as $N\to \infty$. Consequently, the pair $(A_S,\mathcal{C}_f)$ is observable.
The observability of the limit $(A_S,\mathcal{C}_f)$ implies that
the observability matrix $\mathcal{O}_f=\col\{C_f,C_fA_S,\cdots, C_f A_S^{N_1-1} \}$ has full rank. Since $\mathcal{C}_{N_1}$ converges to $\mathcal{C}_f$, by continuity,
 $\mathcal{O}_{N_1}=\col\{C_{N_1},C_{N_1}A_S,\cdots, C_{N_1}A_S^{N_1-1} \}$ converges to $\mathcal{O}_f$ as $N\to \infty$. It follows $0<\|\mathcal{O}_{N_1} ^{-1}\|< \epsilon$ for some $\epsilon>0$ independent of $N$.

For the system $(A_S, \mathcal{C}_{N_1})$, pole placement (Ackermann's formula) yields a gain $\tilde{\mathbb{L}}_1$ satisfying $\|\tilde{\mathbb{L}}_1\| \leqslant M_L$, where $M_L$ depends only on $\epsilon$, $A_S$, and the desired pole locations. 
The original observer gain is obtained by transforming back:
\[
\mathbb{L} = \mathbb{T} \tilde{\mathbb{L}} = 
\begin{bmatrix} I & 0 \\ \mathcal{K} & I \end{bmatrix}
\begin{bmatrix} \tilde{\mathbb{L}}_1 \\ 0 \end{bmatrix} = 
\begin{bmatrix} \tilde{\mathbb{L}}_1 \\ \mathcal{K} \tilde{\mathbb{L}}_1 \end{bmatrix}.
\]
Its norm satisfies:
\[
\|\mathbb{L}\| \leqslant \|\tilde{\mathbb{L}}_1\| + \|\mathcal{K} \tilde{\mathbb{L}}_1\| \leqslant 
\|\tilde{\mathbb{L}}_1\| (1 + \|\mathcal{K}\|) \leqslant M_L (1 + M_1),
\]
which is independent of $N$.


It remains to show the second point. Since $\|\mathcal{K}\|\leqslant M_1$ for some $M_1>0$  independent of $N$, it is easy to check that $1\leqslant \|\mathbb{T}\| \|\mathbb{T}^{-1}\|\leqslant (1+M_1)^2 $. 
It suffices to show      
    $\|e^{(\tilde{A}+\tilde{\mathbb{L}} \tilde{C}+\delta I)t}\| \leqslant M e^{-\delta_1}
    $ for some $M>0$ and $\delta_1>0$. From the diagonal structure of ${A^{N-N_0}}$, we know $\|{A^{N-N_0}}+\delta I\|\leqslant M_2 e^{-\delta_2 t} $  for some $M_2 >0$ and $ \delta_2>0 $ independent of $N $. Note that $\|\tilde{\mathbb{L}}_1C^{N-N_0}\|<M_3 $ for some $M_3>0$  independent of $N $. According to Lemma \ref{proposi:F2} and \eqref{triangularMa}, we just need to show 
    \begin{equation}\label{ex1}
    	\|e^{(A_S + \tilde{\mathbb{L}}_1 \mathcal{C}_{N_1} +\delta I)t}\|\leqslant M_4 e^{-\delta_4 t}
    \end{equation}
        for some positive $M_4$ and $\delta_4$ independent of $N$.
        Note that $\mathcal{C}_{N_1}$ converges to $\mathcal{C}_{f}$ as $N\to \infty $ and $\tilde{\mathbb{L}}_1$ continuously relies on $\mathcal{C}_{N_1}$. It follows that there exist $N^*>0$ such that for any $N>N^*$, \eqref{ex1} holds.  
        This completes the proof.\qedopen
\end{proof}

Consider the closed-loop system \eqref{closedMod} with notations \eqref{eq:system_matrices}.
To state the main result, denote
\begin{equation}\label{notation}
\resizebox{0.98\linewidth}{!}{$
\displaystyle 
\begin{aligned}
M_{11} &= P F\! +\! F^\top P\! +\! 2\delta I 
+ \frac{\alpha( N-1)}{\pi^2 p_m (N-1)^2\! -\! q_M}b_M^2 \bar{K}^T \bar{K}, \\
M_{12} &= P\,\mathcal{L}, \quad  
M_{22} = 2\big(-\lambda_{N+1}  + \delta + \tfrac{1}{2\alpha}\lambda_{N+1}\big)\,\|c\|_N^{-2}, \\
 M_{33} & =-\frac{\alpha \|c\|_N^2} {\lambda_{N+1}}, \quad   \bar{M}_{22}=2\big(-\lambda_{N+1}  + \delta \big)\|c\|_N^{-2},\\
 M_{\alpha}&=(1-2\alpha)\lambda_{N+1}+2\alpha \delta.
\end{aligned} $}
\end{equation}
\begin{theorem}\label{thm:main}
Consider the system \eqref{eq:heat2} with the tracking error \eqref{tracking error}, where $c\in Z$ and $z^o(\cdot,0)\in Z$. Given $\delta>0$, choose $N_0$  such that $\lambda_{N_0+1}>\delta$ and fix $N\geqslant N_0+1.$  Let Assumptions \ref{assump:1}--\ref{assump:ApCp} hold. Consider the control law \eqref{controlLaw} based on the observer \eqref{obsererfinite}.

If there exist a matrix $P>0$ and a scalar $\alpha>\frac{1}{2}$ such that the following LMIs are feasible:
\begin{equation} \label{LMI:nonDelay}
\bar{\mathcal{M}}=
\begin{medsize} \begin{bmatrix}
M_{11} & M_{12} & 0 \\
* & \bar{M}_{22} & 1 \\ 
* & * & M_{33}
\end{bmatrix}\end{medsize}<0,\ \ M_\alpha<0.
\end{equation}
then the control law \eqref{controlLaw} employing observer \eqref{obsererfinite} solves the output regulation problem. Specifically, the state $z^o(\cdot,t)$ of \eqref{eq:heat2} converges to the steady-state trajectory in $Z$  exponentially with decay rate $\delta$, and the tracking error $e(t)$ converges to zero with the same rate:
\begin{equation}\label{eq:decay}
   \begin{aligned}
      \|z^0(\cdot,t)-\Pi w(t)\|^2_{L^2}\; \leqslant\; & \alpha_0  e^{-2\delta t}\|z^0(\cdot,0)-\Pi w(0)\|^2_{L^2}, \\
      |e(t)|^2\; \leqslant \; &\alpha_1  e^{-2\delta t} |y(0)-r(0)|^2
   \end{aligned}
\end{equation}
for some $ \alpha_0, \alpha_1 \geqslant 1.$

Moreover, the LMIs are always feasible for a sufficiently large $N$.
\end{theorem}
\begin{openbox}
\textbf{Guideline on finding $N$:}\\
\textbf{Step 1.} Choose the desired decay rate $\delta$, and select $N_0$ such that $\lambda_{N_0+1} > \delta$. Initially set $N = N_0 + 1$. \\
\textbf{Step 2.}  If $N$ satisfies Assumption~\ref{assum:PN}, go to \textbf{Step 3}. Otherwise, increment $N$ to $N+1$ and repeat until Assumption~\ref{assum:PN} holds. \\
\textbf{Step 3.} Design the observer gain \eqref{gain:obsv_ori} based on the \textbf{ODA} (see Section~\ref{sec:obser_constr}). Check the feasibility of the LMI \eqref{LMI:nonDelay} with this $N$. If \eqref{LMI:nonDelay} is infeasible, increment $N$ to $N+1$ and return to \textbf{Step 2}. 
\end{openbox}
\begin{proof}
To prove the result, it sufficient to establish the stability of the closed-loop system \eqref{closedMod}. 
Define the Lyapunov function as 
\begin{equation}\label{ly:V}
	V(t)=X_0^T(t)P X_0(t) + \sum_{n=N+1}^{\infty} z_n^2(t)
\end{equation}
with $P>0$.

Differentiating $V(t)$ along \eqref{closedMod}, we have 
\begin{equation} \label{eq:Lyapu-nonmomal}
	\begin{aligned} 
\dot{V}(t)+& 2\delta V(t)  = X_0^\top \big(P F\! + F^\top P+\! 2\delta I\big) X_0
+\! 2\, X_0^\top P\, \mathcal{L} \,\zeta \\
&+ 2\sum_{n=N+1}^\infty (-\lambda_n+\delta)\, z_n^2
+ 2\sum_{n=N+1}^\infty z_n\, b_n\, \bar{K}X_0(t). 
\end{aligned}
\end{equation} 
Using the logarithmic inequality $\ln(\frac{1+h}{1-h}) \leqslant \frac{2h}{1-h^2}$, and assuming $N$ is large such that $\pi^2 p_m (N-1)^2 - q_M>0$, we derive:
\begin{align} \label{eq:tail_estimate}
    \sum_{n=N+1}^{\infty} \frac{1}{\lambda_n} 
    &\leqslant \int_{N-1}^{\infty} \frac{1}{\pi^2 p_m x^2 - q_M} \, \dd x \notag \\
    &= \frac{1}{2\pi\sqrt{p_m q_M}} \ln \left( \frac{\pi\sqrt{p_m}(N-1) + \sqrt{q_M}}{\pi\sqrt{p_m}(N-1) - \sqrt{q_M}} \right) \notag \\
    &\leqslant \frac{N-1}{\pi^2 p_m (N-1)^2 - q_M}.
\end{align}
Using Young's inequality, we have  for  $\alpha>0$, \vspace{-12pt} 
\begingroup
\small %
\begin{align}
&2\sum_{n=N+1}^\infty z_n\, b_n\, \bar{K} X_0(t) =
2\sum_{n=N+1}^{\infty} \lambda_n^{\frac{1}{2}}z_n(t)
\frac{b_n}{\lambda_n^{\frac{1}{2}}}  \bar{K} X_0(t) \nonumber \\
&\leqslant 
\frac{1}{\alpha} \sum_{n=N+1}^\infty \lambda_n z_n^2(t)
+ \frac{\alpha( N-1)}{\pi^2 p_m (N-1)^2 - q_M}
b_M^2|\bar{K} X_0(t)|^2,
\label{ineq:1}
\end{align} 
\endgroup
\vspace{-12pt}
\begin{sloppypar}
where the last term is an upper bound of $\sum_{n=N+1}^{\infty} \frac{1}{\lambda_n}b_n^2\, |\bar{K} X_0(t)|^2.$
\end{sloppypar}

Moreover,  from monotonicity of $\lambda_n$, $n\geqslant 1$, we have 
\begin{equation}\label{ineq:2}
\begin{split}
    2\sum_{n=N+1}^\infty & \left(-\lambda_n+\delta+\tfrac{1}{2\alpha} \lambda_{n}\right) z_n^2 \\[-2pt]
    & \leqslant 2\left(-\lambda_{N+1}+\delta+\tfrac{1}{2\alpha} \lambda_{N+1}\right)\|c\|_N^{-2} \zeta^2(t),
\end{split}
\end{equation}
where $\|c\|_N^2= \sum_{i=N+1}^\infty c_i^2$,
 provided  $M_\alpha<0$ and $\alpha>\frac{1}{2}$. For any $\alpha>\frac{1}{2}$, $M_\alpha<0$ holds for a sufficiently large $N$.

Let $\eta (t)=\col\{X_0(t),\zeta(t)\}.$ From \eqref{eq:Lyapu-nonmomal}, \eqref{ineq:1} and \eqref{ineq:2}, we have
\[
\dot{V}(t) + 2\delta V(t) 
\leqslant \eta^\top(t) \Mc\eta(t),
\]
where $\mathcal{M}=\sm{M_{11} & M_{12} \\
* & M_{22} }.
 $
If $\Mc<0$, then $V(t)<e^{-2\delta}V(0)$ which indicates the stability of $z$-subsystem of \eqref{eq:ABGaS} on $Z$. Since the tracking error $e(t)$ is the output of the $z$-subsystem and the output operator is bounded, the convergence of $e(t)$ follows from the fact that the output operator $C$ is bounded.

We now show the feasibility of $\Mc<0$ for large $N$. Note that $M_{22}<0$ for large $\lambda_{N}$. 
 By Schur complement, $\Mc<0$ holds iff 
\begin{equation} \label{LMI:Mc}
	 M_{11}- P\Lscr \frac{ \|c\|^2_N}{(-2+\frac{1}{\alpha})\lambda_{N+1}+\delta }{ \Lscr}^\top P <0.
\end{equation}
By the design process, $F+\delta I$ is Hurwitz, and so the Lyapunov equation 
\begin{equation}
	\label{LyapuEq}
	(F+\delta I )^\top P+ P(F+\delta I)=- I,
\end{equation} 
admits a  solution
\begin{equation}\label{solutionOfLyapu}
	P=\int_0^\infty e^{(F^\top+\delta I)  t} e^{(F+\delta I ) t}\dd t. 
\end{equation}
From \eqref{LMI:Mc} and \eqref{notation}, we know that  $M_{11}<0$ holds with the $P$ of \eqref{solutionOfLyapu} for large enough $N$.
Noting  the monotonicity of $\lambda_N$ and $\|c\|_N$, the second term in \eqref{LMI:Mc} tends to zero as $N\rightarrow \infty$ if $\|P\|$ and $\|\mathcal{L}\|$ is bounded independently of $N$. 
According to Lemma \ref{lemma:invarience-gain}, the observer gain $\|L^N\|$ and $\|L^w\|$ are uniformly bounded independent of $N$, so that we obtain $\|\Lscr\|$ in \eqref{eq:system_matrices}  is bounded independent of $N$.

To show  the uniform boundedness $P$ of \eqref{solutionOfLyapu}, it suffices to show that there exist  $\delta_1, M_1>0$ independent of $N$ such that $ \|e^{(F+\delta I)t}\|< M_1 e^{-\delta_1 t}$. 
Denote $F_\delta=F+\delta I$, $F_K=A_K+\delta I, F_L=A_L+\delta I$ and $F_1= [L^N C^N \ \ 0].$ 
Using the block-triangular structure of $F_\delta$, we have
$$
e^{F_\delta t}
= \sm{e^{F_K t} & \ \ \ \int_0^t e^{F_K (t-s)} F_1 e^{F_L s} \dd s \\
0 & e^{F_L t}} .
$$
and
\[
\|e^{F_\delta t}\| \leqslant 
\|e^{F_K t}\| + \|e^{F_L t}\| +
\int_0^t \|e^{F_K (t-s)}\| \|F_1\| \|e^{F_L s}\| \dd s.
\]
Clearly, 
$\|F_1\|\leqslant M_{F_1}$ for some $M_{F_1}>0$ independent of $N$. 
Providing that there exist $\delta_K,\delta_L, M_L, M_K>0$  independent of $N$ such that  $\|e^{(A_L+\delta I) t}\| \leqslant M_L e^{-\delta_L  t} $ and $ \|e^{{(A_K+\delta I)} t}\|\leqslant M_K e^{-\delta_K t} $,
we have 
\begin{align}
	 \|e^{F_\delta t}\| \leqslant &M_K e^{-\delta_K t} + M_L e^{-\delta_L t} \notag \\
    &+ M_K M_L M_{F_1} \int_0^t e^{-\delta_K (t-s)} e^{-\delta_L s} \dd s \label{ineq:mm}
\end{align}
where the integral satisfies
\begin{equation}\label{ineq:int}
	\int_0^t e^{-\delta_K (t-s)} e^{-\delta_L s} \dd s
\leqslant \frac{1}{|\delta_K-\delta_L|} e^{-\min(\delta_K,\delta_L)t}.
\end{equation}
From Lemma \ref{lemma:invarience-gain}, we have $\|e^{(A_L+\delta I) t}\| \leqslant M_L e^{-\delta_L  t} $ for some $\delta_L, M_L>0$  independent of $N$. 
Lemma \ref{proposi:F2} (where we take $F_0=A^{N_0}+B^{N_0}K_0$, $F_2=B^{N-N_0} K_0$, and $N_1=N_0$) further indicates $\|e^{(A_K+\delta I)t}\| \le M_K e^{-\delta_K t},$ for some $\delta_K, M_K>0$ independent of $N$, since both $\|K\|$ and the entries of $B^{N-N_0}K_0$ 
are uniformly bounded in $N$.
 Combined with \eqref{ineq:mm} and \eqref{ineq:int}, we conclude the uniform boundedness of $\|P\|.$    

It can be seen from \eqref{notation} that $M_{11}<0$ holds for large $N$, since 
  $\bar{K}=[K  \ \Gamma \ 0]$  is independent of $N$. 
 This shows $\mathcal{M}<0$ for large enough $N$, which is equivalent to $\bar{\Mc}<0$ after applying Schur complement to $\bar{\mathcal{M}}$. 
\qedopen
\end{proof}

\begin{remark} \label{Rem:noval}
In finite-dimensional observer designs meant only for the stabilization of linear parabolic systems (e.g., \cite{KATZ2020109285,katz2021finite,katz2022delayed}), the observer gain typically depends only on the first $N_0$ dominant modes. In such cases, this truncation is sufficient to guarantee that the closed-loop system matrix $F$ defined by \eqref{eq:system_matrices} is Hurwitz. 

For the output regulation problem considered here, however, restricting the controller design to the first $N_0$ modes is inadequate. 
Unlike the pure stabilization case, the presence of the term $B^{N-N_0}\Gamma$ can render $A_L$ in \eqref{eq:system_matrices} (and hence $F$) unstable if the gain only relies on the first $N_0$ modes and the exosystem $S$. 

This motivates our approach, where the observer gain has full order. This $N$-dependency introduces a major theoretical challenge: the LMI feasibility analysis methods from existing stabilization frameworks are no longer applicable. Consequently, we must develop a novel feasibility analysis and prove that the norm of the observer gain remains uniformly bounded independent of $N$.  This new analysis provides a fresh perspective for the LMI feasibility analysis.
\end{remark}

\section{Output regulation in the presence of measurement delay}
\label{sec:delay_regulation}
We consider the output regulation problem with a time-varying delay in the measurement. 
The system dynamics are given by \eqref{eq:heat2}, the reference signal $r(t)$ and the disturbances $\bar{d}_i(t)$ are generated by a linear finite-dimensional system \eqref{eq:exosystem1}, and the regulated output is given by \eqref{eq:3.2}. 
Meanwhile, the available \textit{measurement} is the delayed tracking error
\begin{equation}\label{measurementDel}
    \begin{aligned}
          & e(t-\tau_y(t)) = \int_0^1 z(x, t-\tau_y(t)) c(x) \dd x, \\ 
    & 0 \leqslant \tau_y(t) \leqslant \tau_M,
    \end{aligned}
\end{equation} 
with the initial condition $e(t-\tau_y(t)) = 0$ for $t-\tau_y(t) \leqslant 0$. 
Regarding the delay $\tau_y(t)$, we consider two cases: 
\begin{enumerate}[{label=\arabic*)}]
    \item \textit{Differentiable delay:} $\tau_y(t)\geqslant \tau_m>0$ is continuously differentiable, where the lower bound $\tau_m$ is required only for the well-posedness of the PDE. 
    \item \textit{Sawtooth delay:} $\tau_y(t) = t - t_k$ for $t \in [t_k, t_{k+1})$, 
    with $lim_{k\to \infty} t_k=\infty.$
\end{enumerate}

To solve this problem, we rely on the finite-dimensional truncated model \eqref{eq:ABGaS} derived in Section 3.2. Since the instantaneous tracking error $e(t)$ is unavailable, we adapt the observer design to utilize the available delayed measurement \eqref{measurementDel}. 
Specifically,  we employ the same finite-dimensional observer-based controller structure \eqref{obsererfinite}, but driven by the delayed measurement.
The observer is given by:
\begingroup            
\begin{equation}\label{obser:del}
\begin{split}
    \frac{\dd}{\dd t} \left[\begin{smallmatrix} \hat{w}(t) \\ \hat{z}^N(t) \end{smallmatrix}\right] &= 
    \left[\begin{smallmatrix} S & 0 \\ -B^N \Gamma & A^N \end{smallmatrix}\right] 
    \left[\begin{smallmatrix} \hat{w}(t) \\ \hat{z}^N(t) \end{smallmatrix}\right] 
    + \left[\begin{smallmatrix} 0 \\ B^N \end{smallmatrix}\right] u(t) \\
    &\quad + \left[\begin{smallmatrix} L^w \\ L^N \end{smallmatrix}\right] \big( C^N \hat{z}^N(t) - e(t-\tau_y(t)) \big),
\end{split}
\end{equation} 
\endgroup
 with the control law \eqref{controlLaw}, as before,
where the gains $K$ and $(L^w, L^N)$ are designed the same as in Section \ref{sec:obser_constr}.
We now derive the closed-loop error dynamics. Let $X_0(t)$ be the closed-loop state vector defined as in \eqref{eq:def_X0}. To account for the measurement delay, we introduce the difference term:
\begin{equation}
    v_{\tau}(t) = X_0(t) - X_0(t-\tau_y(t)).
\end{equation}
Additionally, define the delay-induced perturbation matrix as $\mathcal{L}_2 = \mathcal{L} \begin{bmatrix} C^N & 0 & -C^N \end{bmatrix}$. 
Based on \eqref{obsererfinite}, \eqref{obser:del}, \eqref{controlLaw} and \eqref{eq:def_X0}, the closed-loop system is summarized as:
\begin{subequations}\label{eq:clo-OutputDelay}
    \begin{align}
        \dot{X}_0(t) &= F X_0(t) + \mathcal{L} \zeta(t-\tau_y(t)) - \mathcal{L}_2 v_{\tau}(t), \label{eq:closed_loop_X0} \\
        \dot{z}_n(t) &= -\lambda_n z_n(t) + b_n \bar{K} X_0(t), \quad n \ge N+1, \label{eq:closed_loop_zn}
    \end{align}
\end{subequations}
where $F$, $\mathcal{L}$, $\bar{K}$ and $\zeta$ are defined in \eqref{eq:system_matrices}.

 Following arguments in \cite{KATZ2021109364}, for differentiable delay, assume that there exists a unique $t_*\in[\tau_m, \tau_M ]$ such that $t_*=\tau_y(t_*)$. For any initial value $z_0\in \mathcal{D}(A)$, the closed-loop system \eqref{eq:clo-OutputDelay} 
 admits a unique solution satisfying $z\in C([0,\infty), L^2(0,1)).$ 
 For sawtooth delay, well-posedness is similarly guaranteed by the arguments in \cite{KATZ2021109364}. 

To compensate for the term $\zeta(t-\tau_y(t))$, we need the following lemma:
\begin{lemma}[Halanay's inequality, \cite{MR3237720}]\label{lem:halanay}
Let $0 < \delta_1 < \delta_0$ and let 
$V_h : [t_0 - \tau_M, \infty) \to [0,\infty)$
be an absolutely continuous function that satisfies
\[
\dot V_h(t) + 2\delta_0 V_h(t) - 2\delta_1 
\sup_{-\tau_M \le \theta \le 0} V_h(t+\theta) \leqslant 0, 
\qquad t \geqslant t_0.
\]
Then 
$V_h(t) \leqslant \exp\!\left( -2\delta_{\tau_M} (t - t_0) \right)
\sup_{-\tau_M \leqslant \theta \leqslant 0} V_h(t_0+\theta), 
\quad t \geqslant t_0,$
where $\delta_{\tau_M} > 0$ is the unique positive solution of
$
\delta_{\tau_M} = \delta_0 - \delta_1 \exp(2\delta_{\tau_M}\tau_M).
$
\end{lemma}
To state the result, denote
\begin{equation}\label{eq:matrix_H_def}
\setlength{\jot}{2pt}
\begin{aligned}
    &\bar{H} = \left[ \begin{smallmatrix} H_0 & \Lambda_1 \\ * & H_1 \end{smallmatrix} \right], \quad
    H_0 = \left[ \begin{smallmatrix} H_{11} & H_{12} \\ * & H_{22} \end{smallmatrix} \right], \quad
    H_1 = \left[ \begin{smallmatrix} H_{33} & H_{34} \\ * & H_{44} \end{smallmatrix} \right], \\
    &\Lambda_1 = \left[ \begin{smallmatrix} H_{13} & H_{14} \\ 0 & 0 \end{smallmatrix} \right], \quad
    \mathcal{H}_{\tau_M} = \Phi_1^\top R_1 \Phi_1, \quad \varepsilon_M = e^{-2 \delta_0\tau_M }, \\
   &H_{11} = M_{11}+ (1 - \varepsilon_M) S_1, \ H_{12} = P\Lscr, \ H_{22} = -\frac{2\delta_1 }{\|c\|_N^2}, \\
    &H_{13} = -P\Lscr_2 + 2\delta_1 P + \varepsilon_M S_1, \quad H_{14} = \varepsilon_M S_1, \\
    &H_{33} = -2\delta_1 P - \varepsilon_M (S_1 + R_1), \quad H_{34} = -\varepsilon_M (S_1 + G_1), \\
    &H_{44} = -\varepsilon_M (S_1 + R_1), \quad \Phi_1 = [\, F \;\; \Lscr \;\; -\Lscr_2 \;\; 0 \,], \\
    &\theta_{\tau}(t) = X_0(t-\tauy)- X_0(t-\tau_M).
\end{aligned}
\end{equation}

\begin{theorem}\label{thm:main2}
Consider the system \eqref{eq:heat2} with the measurement \eqref{measurementDel}, where $c\in Z$ and $z^o(\cdot,0)\in Z$. With the same assumptions as Theorem \ref{thm:main},  consider the control law \eqref{controlLaw} based on the observer \eqref{obser:del}. 

Let $\delta_0>\delta>0$ and define $\delta_1=\delta_0-\delta$. Given $\tau_M>0$, suppose there exist positive definite matrices $P, S_1, R_1$, a scalar $\alpha>\frac{1}{2}$, and a matrice $G_1$ such that the following LMIs hold:
\begin{equation}\label{eq:LMIdelay}
\begin{bmatrix} R_1 & G_1 \\ * & R_1 \end{bmatrix} \geqslant 0, \ \
\bar{H}+\tau_M^2 \Phi_1^\top R_1 \Phi_1 \leqslant 0, \ \ M_\alpha<0.
\end{equation}
Then, this controller 
solves the output regulation problem. Specifically, the state $z^o(\cdot,t)$ of \eqref{eq:heat2} converges to the steady-state trajectory in $Z$ exponentially with  decay rate $\delta_{\tau_M}$, and the tracking error $e(t)$ converges to zero with the same rate, i.e., \eqref{eq:decay} holds with $\delta$ changed by $\delta_{\tau_M}$.

Moreover, the LMI conditions \eqref{eq:LMIdelay} are always feasible for a sufficiently large $N$ and a sufficiently small $\tau_M$.
\end{theorem}
The proof is presented in the Appendix.

\section{Numerical examples}

To validate the theoretical results and the proposed regulator design, we consider a numerical example where the objective is to reject a sinusoidal distributed disturbance while tracking a constant reference signal. The finite difference method is used for the simulation.

The diffusion and reaction coefficients of system \eqref{eq:heat2} are $p(x) = 1$ and $q(x) = \pi^2$. The output is measured over a partial domain $[0, \bar{c}]$ with $\bar{c} = \frac{1}{2\sqrt{2}}$.

The disturbance signals are $\bar{d}_1(t) = \sin(3t)$, $\bar{d}_2(t)=0$, and $\bar{d}_3(t)=0$, with the location parameter $g(x)=1$. The output $y(t)$ is required to track a constant setpoint $r(t) = 1$. A numerical result shows that $|\PPP(0)|\approx 0.411$ and $|\PPP(i3)|\approx 0.123,$ so Assumption \ref{assump:1} holds.

The exosystem is given by $\dot{w} = Sw$ with state $w = [w_1, w_2, w_3]^T$. The first two states generate the sinusoidal signal $\sin(3t)$, and the third state generates the constant reference signal. The system matrices are defined as:
\begin{equation}
    S = \sm{
    0 & 3 & 0 \\ 
    -3 & 0 & 0 \\ 
    0 & 0 & 0 
}, \quad 
    w(0) = \sm{ 0 \\ 1 \\ 1 }.
\end{equation} 
Accordingly, the disturbance distribution vector $d_1$ and the reference mapping vector $Q$ are set as:
\begin{equation}
    d_1 = [1, \ 0, \ 0], \quad Q = [0, \ 0, \ -1].
\end{equation}
Solving the regulator equations, we obtain
\begin{equation}
    \Gamma \approx \begin{bmatrix} -0.85 & -0.41 & -2.43 \end{bmatrix}.
\end{equation}
It is easy to see $(S,\Gamma)$ is observable. 
A formal calculation gives the eigenpair $(\lambda_n, \phi_n)$ of the operator $-{A}$:  $\lambda_1= -1$, $\phi_1=1$ and 
 $\lambda_n= (n-1)^2\pi^2-1$, $\phi_n(x)=\sqrt{2}\cos((n-1)\pi x) $, $n= 2, 3,\cdots .$  The open-loop system \eqref{eq:heat2} is unstable since $-\lambda_1>0$.
We have $A^N=\diag\{-\lambda_n \}_{n=1}^N$, $B^N=\col\{ \phi_n (1)\}_{n=1}^N
 $ and $C^N=[c_1, c_1,\cdots, c_N ], c_1=\frac{1}{2\sqrt{2}}, c_n=\frac{\sqrt{2}}{(n-1)\pi }\sin (\frac{(n-1)\pi }{2\sqrt{2}}), n=2,3,\cdots, N. $  Clearly, $b_nc_n\neq 0$ for all $n\geqslant 1.$
 
 Fig.~\ref{fig1:main_figure}(left) shows that Assumption \ref{assum:PN} holds for all $N>0$.
By Lemma \ref{lem:observability}, \eqref{eq:ABGaS} is observable  for all $N>0$.  
 
 We select a desired decay rate of $\delta = 2$.  With this $\delta$, set $N_0=1$ so that $\lambda_{N_0+1}=\pi^2-1>\delta$. 
 We choose $K_0=-3.15$ such that $A^{N_0}+B^{N_0}K_0=-2.15<-\delta$. 
 
  To find $N$, we follow the \textbf{Guideline} introduced after Theorem \ref{thm:main}.  
  The observer gain is designed based on the \textbf{ODA}, where $N_1=N_0+n_w=4$ and $\tilde{\mathbb{L}}_1 $ is designed by Ackermann's formula 
   such that $$\sigma(A_S +\tilde{\mathbb{L}}_1 (\mathcal{C}_1 + C^{N-N_0} \mathcal{K}))=\{-2.1, -2.2, -2.3, -2.4 \}. $$
  By iterations, the dimension derived via LMI to guarantee this decay rate is $N=3$. The obtained gain is $\col\{L^w, L^N\}\approx\col\{ -42.1,
   50.1, -2.8, -36.2,  7.6, -1.4\}. $

  Fig.~\ref{fig1:main_figure}(right) confirms that the observer gain norm remains bounded as $N$ increases. By setting $N=3$, the resulting tracking performance, control input, and state evolution $z(x,t)$ are presented in Fig.~\ref{fig2:main_figure}. In the end, we introduce a constant measurement delay $\tau_y = 0.05$ while keeping other parameters unchanged; the tracking performance and state response are shown in Fig.~\ref{fig4:main_figure}.
\begin{figure}[htbp]
    \centering
    \begin{minipage}[b]{0.48\linewidth}
        \centering
        \includegraphics[width=\linewidth]{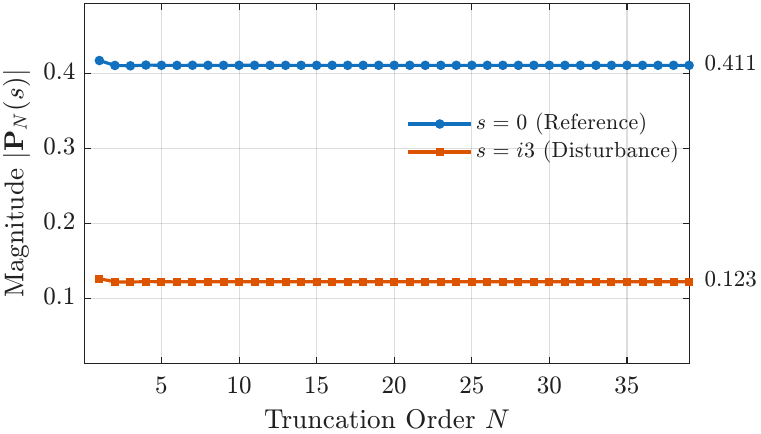}
    \end{minipage}
    \hfill 
    \begin{minipage}[b]{0.48\linewidth}
        \centering
        \includegraphics[width=\linewidth]{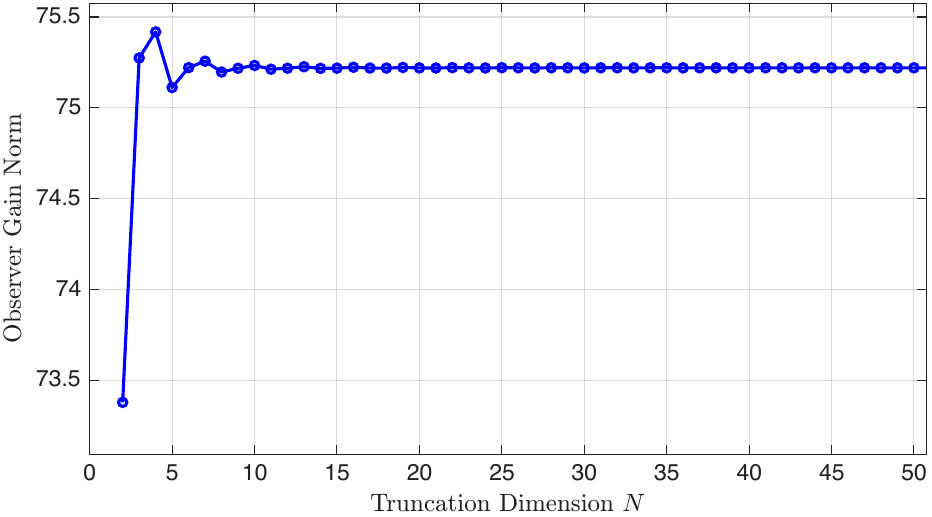}
    \end{minipage}
    
    \caption{Magnitude of $\mathbf{P}_N(s)$ (left) and Observer gain (right).}
    \label{fig1:main_figure}
\end{figure}
\begin{figure}[htbp]
    \centering
    \begin{minipage}[b]{0.48\linewidth}
        \centering
        \includegraphics[width=\linewidth]{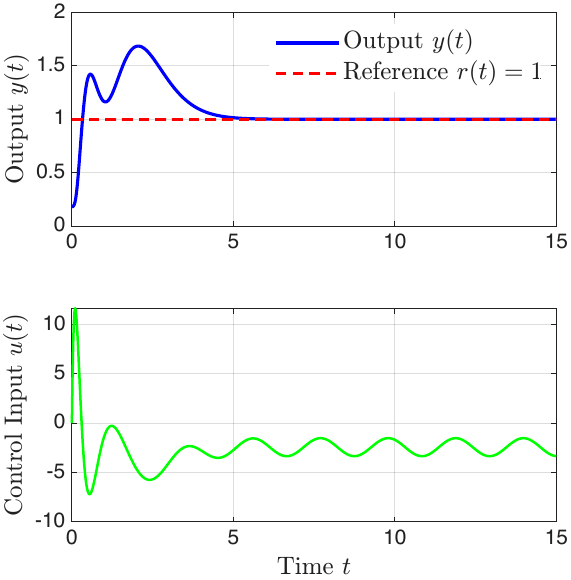}
    \end{minipage}
    \hfill
    \begin{minipage}[b]{0.48\linewidth}
        \centering
        \includegraphics[width=\linewidth]{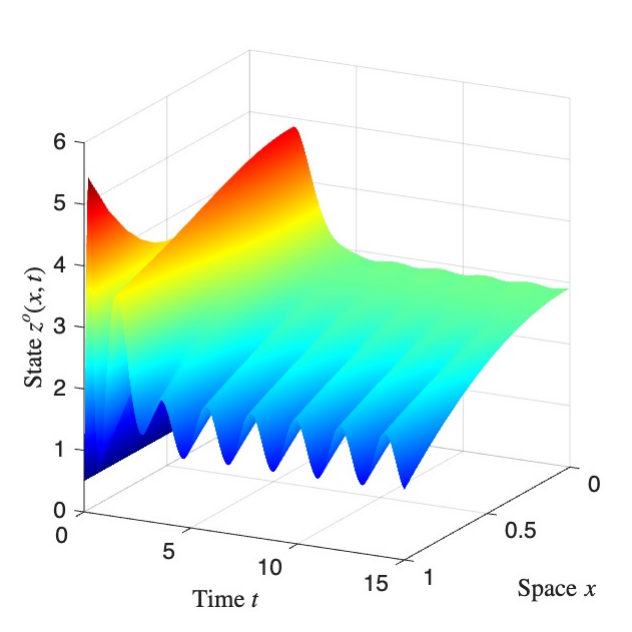}
    \end{minipage}
    \caption{Tracking performance and control input (left); State evolution $z(x,t)$ (right).}
    \label{fig2:main_figure}
\end{figure}

\begin{figure}[htbp]
    \centering
    \begin{minipage}[b]{0.48\linewidth}
        \centering
        \includegraphics[width=\linewidth]{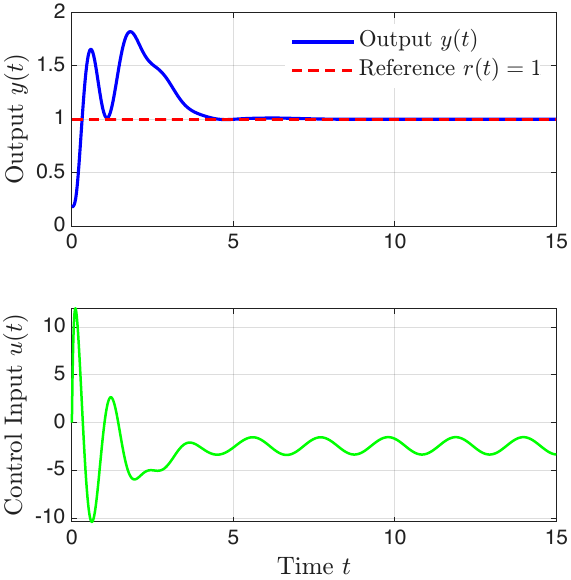}
    \end{minipage}
    \hfill
    \begin{minipage}[b]{0.48\linewidth}
        \centering
        \includegraphics[width=\linewidth]{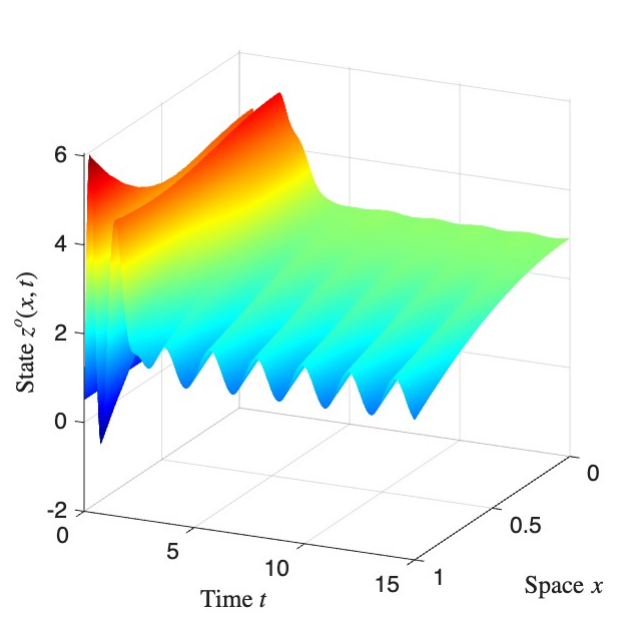}
    \end{minipage}
    \caption{Tracking performance, control input (left) and state evolution (right) with constant measurement delay $\tau_y=0.06$.}
    \label{fig4:main_figure}
\end{figure}

\section{Conclusion}
In this paper, we proposed a finite-dimensional regulator design using tracking error feedback for a class of parabolic systems with Neumann boundary actuation, non-local measurements, and unknown time-varying measurement delays.  It is also possible to extend the proposed method to consider unbounded measurements and other boundary inputs (e.g., \cite{katz2021finite}),  which needs further investigation.
\section*{Appendix}

\begin{myproof}{Proposition \ref{propo:solvabilityOfRegu}}
The regulator equation for $\Pi(x)$ is a linear non-homogeneous matrix ODE. Rearranging \eqref{eq:regulator1}:
\begin{equation} \label{eq:matrix_ode}
    -(p(x)\Pi'(x))' + \Pi(x)(S - q(x)I) = g(x)d_1, \quad \Pi'(0) = d_2. 
\end{equation}
Let $\Psi(x) \in \mathbb{R}^{\n \times \n}$ be the matrix fundamental solution satisfying the homogeneous equation $-(p\Psi')'(x) + \Psi(x)(S - q(x)I) = 0$ subject to $\Psi(0)=I$ and $\Psi'(0)=0$. Let $\Pi_p(x)$ be a particular solution to \eqref{eq:matrix_ode} with zero initial conditions $\Pi_p(0)=0$.
The general solution to \eqref{eq:matrix_ode} can be uniquely expressed as:
\begin{equation}
    \Pi(x) = \mathbf{K}\Psi(x) + \Pi_p(x),
\end{equation}
where $\mathbf{K} = \Pi(0) \in \mathbb{R}^{1 \times \n}$ is the unknown initial vector to be determined. 

Substituting the general solution into the zero-error constraint $\int_0^1 c(x)\Pi(x)\dd x=- Q$,  we obtain
\begin{equation}
    \mathbf{K} \left( \int_0^1 c(x)\Psi(x)\dd x \right) = -Q - \int_0^1 c(x)\Pi_p(x)\dd x.
\end{equation}
Define the moment matrix $\mathbf{M} := \int_0^1 c(x)\Psi(x)\dd x$. The existence of a unique solution pair $(\Pi, \Gamma)$ is equivalent to the existence of a unique vector $\mathbf{K}$, which holds if and only if $\mathbf{M}$ is invertible, i.e., $\det(\mathbf{M}) \neq 0$.

We now relate the spectrum of $\mathbf{M}$ to the transfer function zero dynamics. Let $(\lambda, v)$ be an eigenpair of the exosystem matrix $S$, such that $S v = \lambda v$.
Consider the vector function $y(x) := \Psi(x)v$. Applying the differential operator to $y(x)$ and using the homogeneous ODE of $\Psi(x)$, we have:
\begin{equation}
    -(py')' + (\lambda - q)y = 0, \quad \text{with } y(0)=v, \ y'(0)=0.
\end{equation}
This is identical to the IVP \eqref{eq:scalar_phi} defining the scalar fundamental solution $\phi(x, \lambda)$ scaled by the vector $v$. By the uniqueness of ODE solutions, we have $\Psi(x)v = \phi(x, \lambda)v$.

Applying $\mathbf{M}$ to the eigenvector $v$ yields:
\begin{equation} \label{eq:eigen_mapping}
    \mathbf{M}v = \left( \int_0^1 c(x)\Psi(x)\dd x \right) v  = \left( \int_0^1 c(x)\phi(x, \lambda)\dd x \right) v.
\end{equation}
Recall that the numerator of the transfer function is $\mathbf{N}(\lambda) = \int_0^1 c(x)\phi(x, \lambda)\dd x$. Equation \eqref{eq:eigen_mapping} implies that $\mathbf{N}(\lambda)$ is an eigenvalue of $\mathbf{M}$ corresponding to the eigenvector $v$.

By the spectral mapping theorem, we have:
\begin{equation}
    \det(\mathbf{M}) = \prod_{\lambda \in \sigma(S)} \mathbf{N}(\lambda).
\end{equation}
Therefore, $\mathbf{M}$ is invertible if and only if $\mathbf{N}(\lambda) \neq 0$ for all $\lambda \in \sigma(S)$. \end{myproof}
\begin{myproof}{Proposition \ref{prop:operator-version}}
Since $(\mathcal{A}^p, C^p)$ is not approximately observable in infinite time, there exists a nonzero unobservable subspace
\[
\Ns = \left\{ x^p \in X \mid C^p T^p(t) x^p = 0, \ \forall t \geqslant 0 \right\},
\]
where $T^p(t)$ is the semigroup generated by $\mathcal{A}^p$.

Notice that $\Ns \cap (Z \times \{0\}) = \{0\}$. Indeed, if $(z, 0) \in \Ns$, then $C T(t)z = 0$ for all $t \geqslant 0$, so by the approximate observability of $(A, C)$, $z = 0$.

Let $\pi_W : X \to W$ be the projection onto $W$. The restriction $\pi_W|_\Ns$ is injective because if $(z_1, w)$ and $(z_2, w)$ are in $\Ns$, then $(z_1 - z_2, 0) \in \Ns$, so $z_1 = z_2$. Define
$
W_2 = \pi_W(\Ns) \subseteq W,
$
which is a $S-$invariant finite-dimensional subspace of $W$. Let $W_1$ be a complement of $W_2$ in $W$, so that $W = W_1 \oplus W_2$.

By the injectivity of $\pi_W|_\Ns$, for each $w_2 \in W_2$, there exists a unique $z \in Z$ such that $(z, w_2) \in \Ns$. Define the linear map $L : W_2 \to Z$ by $L(w_2) = z$. Then
$
\Ns = \left\{ (L w_2, w_2) \mid w_2 \in W_2 \right\}.
$
Clearly,  $L$ is bounded.

Now we show that $\Ns \subset D(\mathcal{A}^p)$ and $\mathcal{A}^p \Ns \subset \Ns$. Since $\Ns$ is finite-dimensional and $T^p(t)$-invariant, the restriction $T^p(t)|_\Ns$ is a strongly continuous semigroup on  $\Ns$. On finite-dimensional spaces, all strongly continuous semigroups are uniformly continuous and generated by a bounded linear operator $A_\Ns : \Ns \to \Ns$. 
Clearly, $A_\Ns$ must be the restriction of $A$ to $\Ns$, hence
 $\Ns \subset D(\mathcal{A}^p)$ and $\mathcal{A}^p \Ns \subset \Ns$.
 Expressing $S$ in block form with respect to the decomposition $W = W_1 \oplus W_2$:
$
S = \begin{bmatrix} S_{11} & 0 \\ S_{21} & S_{22} \end{bmatrix},
$
where
$S_{11}: W_1 \to W_1$,
$S_{21}: W_1 \to W_2$,
$S_{22}: W_2 \to W_2$.
For any $w_2\in W_2$, we have 
\begin{equation}\label{AP}
	\mathcal{A}^p \bigl[\begin{smallmatrix} L w_2 \\ w_2 \end{smallmatrix}\bigr]  = \bigl[\begin{smallmatrix}A L w_2 + P w_2 \\ S w_2 \end{smallmatrix} \bigr]\in \Ns,
\end{equation}
which implies 
$
A L w_2 + P w_2 = L (S w_2) = L (S_{22} w_2).
$
Decomposing $P$ as $P = [P_1\  P_2]$ with respect to $W = W_1 \oplus W_2$, we have $P w_2 = P_2 w_2$, so
\begin{equation}\label{AL}
	A L w_2 + P_2 w_2 = L S_{22} w_2 \quad \forall w_2 \in W_2. 
\end{equation}
For any $w_2\in W_2$, we have
\[
C^p \sm{L w_2\\ w_2}  = C L w_2 + Q w_2 = 0.
\]
Decomposing $Q$ as $Q = [Q_1\  Q_2]$ with respect to $W = W_1 \oplus W_2$, we have $Q w_2 = Q_2 w_2$, so
$
C L w_2 + Q_2 w_2 = 0, \forall w_2 \in W_2,
$
which implies
\begin{equation}\label{CLQ}
	C L + Q_2 = 0. 
\end{equation}
Define the transformation $\mathcal{R} : X \to X$ by
\begin{equation}
	\mathcal{R}(z, w_1, w_2) = (z - L w_2, w_1, w_2).
\end{equation}
This is bounded and invertible, with inverse
\[
\mathcal{R}^{-1}(\tilde{z}, \tilde{w}_1, \tilde{w}_2) = (\tilde{z} + L \tilde{w}_2, \tilde{w}_1, \tilde{w}_2).
\]
In the new coordinates, the unobservable subspace becomes
\[
\tilde{\Ns} = \mathcal{R}(\Ns) = \left\{ (0, 0, w_2) \mid w_2 \in W_2 \right\}.
\]
Now compute the transformed operators. For the input operator, we have
$
\tilde{B}^p = \mathcal{R} B^p = \begin{psmallmatrix} B \\ 0 \end{psmallmatrix}.
$
For the output operator, let $\tilde{x}^p = (\tilde{z}, \tilde{w}_1, \tilde{w}_2)$. Then
$$
\tilde{C}^p \tilde{x}^p = C^p \mathcal{R}^{-1} \tilde{x}^p = C(\tilde{z} + L \tilde{w}_2) + Q_1 \tilde{w}_1 + Q_2 \tilde{w}_2.
$$
Using \dref{CLQ}, $C L + Q_2 = 0$, we get
$
\tilde{C}^p \tilde{x}^p = C \tilde{z} + Q_1 \tilde{w}_1,
$
so
$
\tilde{C}^p = (C \ Q_1 \ 0) = (C \ \tilde{Q}), $ where $\tilde{Q} = (Q_1, 0).
$
For the state operator,
$
\mathcal{R}^{-1} \tilde{x}^p = (\tilde{z} + L \tilde{w}_2, \tilde{w}_1, \tilde{w}_2),
$
and since $S_{12} = 0$, we have
\begin{equation}\label{AT}
	\mathcal{A}^p \mathcal{R}^{-1} \tilde{x}^p =\begin{medsize} \begin{bmatrix}
A \tilde{z} + A L \tilde{w}_2 + P_1 \tilde{w}_1 + P_2 \tilde{w}_2 \\
S_{11} \tilde{w}_1 \\
S_{21} \tilde{w}_1 + S_{22} \tilde{w}_2
\end{bmatrix}
\end{medsize}.
\end{equation}
Applying the transformation $\mathcal{R}$ to \dref{AT} and utilizing the relation \dref{AL}, we obtain the triangular structure:
\begin{equation}
    \tilde{A}^p = \mathcal{R} \mathcal{A}^p \mathcal{R}^{-1} = 
   \begin{medsize} 
\begin{bmatrix}
    A & P_1 - L S_{21} & 0 \\
    0 & S_{11} & 0 \\
    0 & S_{21} & S_{22}
\end{bmatrix}
\end{medsize}.
\end{equation}
It follows that the system dynamics satisfy \[\tilde{A}^p \tilde{x}^p = \col\{ A \tilde{z} + (P_1 - L S_{21}) \tilde{w}_1, S_{11} \tilde{w}_1, S_{21} \tilde{w}_1 + S_{22} \tilde{w}_2 \}.\]
Let $\tilde{P} = [\tilde{P}_1\ 0]$ and  $\tilde{P}_1=P_1 - L S_{21}$. Then
\[
\tilde{A}^p = \begin{bmatrix} A & \tilde{P} \\ 0 & \tilde{S} \end{bmatrix}, \quad 
\tilde{S} = \begin{bmatrix} S_{11} & 0 \\ S_{21} & S_{22} \end{bmatrix}, \quad 
\tilde{P} = [\tilde{P}_1\ 0].
\]
Finally, we show the reduced subsystem $\big( [ \begin{smallmatrix} A & P_1 \\ 0 & S_{11} \end{smallmatrix} ], [C \; Q_1] \big)$ is approximately observable in infinite time. Suppose  $(z, w_1)\neq 0$ is unobservable for this subsystem, i.e.,
\[
(C \;\; Q_1) T_{\text{red}}(t) \begin{psmallmatrix} z \\ w_1 \end{psmallmatrix} = 0 \ \quad \forall t \geqslant 0,
\]
where $T_{\text{red}}(t)$ is the semigroup generated by $\begin{psmallmatrix} A & P_1 \\ 0 & S_{11} \end{psmallmatrix}$. Consider the state $(z, w_1, 0)$ in the transformed full system. Its output is:
\[
\tilde{C}^p T^p(t) (z, w_1, 0) = (C \ Q_1) T_{\text{red}}(t) \begin{psmallmatrix} z \\ w_1 \end{psmallmatrix} = 0 \ \quad \forall t \geqslant 0,
\]
so $(z, w_1, 0) \in \tilde{\Ns}$. But $\tilde{\Ns} = \{ (0, 0, w_2) \mid w_2 \in W_2 \}$, so $z = 0$ and $w_1 = 0$. Hence, the reduced subsystem is approximately observable.
\end{myproof}
\begin{myproof}{Proposition \ref{propo:SGobservability}}
Assume, by contradiction, that $(S, \Gamma)$ is not observable. Since $S$ is a matrix with $\sigma(S) \subset \mathbb{C}_+$, there exists $\lambda \in \sigma(S)$ with $\Re \lambda \geqslant 0$ and a corresponding eigenvector $w \neq 0$ such that $Sw = \lambda w$ and $\Gamma w = 0$.

Consider $x_0 = \sm{\Pi w \\ w }$.
 Applying the first equation of \eqref{eq:regulator} to $w$ gives
\[
\Pi Sw = A(\Pi w) + B(\Gamma w) + Pw.
\]
Since $Sw = \lambda w$ and $\Gamma w = 0$, this simplifies to:
$
\lambda \Pi w = A(\Pi w) + Pw.
$
Therefore, $A(\Pi w) + Pw = \lambda \Pi w \in Z$, which implies that $x_0 \in \mathcal{D}(\mathcal{A}^p)$. Then 
\[
\mathcal{A}^p x_0 
= \begin{pmatrix} A\Pi w + Pw \\ Sw \end{pmatrix} = \begin{pmatrix} \lambda \Pi w \\ \lambda w \end{pmatrix} = \lambda x_0.
\]
From the second equation of \eqref{eq:regulator},
$
C^px_0 = (C\Pi + Q)w = 0.
$
Thus, $x_0$ is an eigenvector of $\mathcal{A}^p$ with eigenvalue $\lambda$ and $C^p x_0 = 0$, contradicting the approximate observability of $(\mathcal{A}^p, C^p)$. Hence, $(S, \Gamma)$ is observable.
\end{myproof}

\begin{myproof}{Theorem \ref{thm:main2}}
Following \cite{katz2022delayed}, define the Lyapunov functional:
\begin{align}\label{Lya:vh}
	V_h(t)= V(t)+ V_{S_1}(t)+ V_{R_1}(t),
\end{align}
where $V(t)$ is defined in \dref{ly:V}, and
\begin{align*}
	&V_{S_1}(t)= \int_{t-\tau_M}^{t} e^{-2\delta_0 (t-\tau)}\| X_0(\tau)\|_{S_1}^2 \dd \tau, \\
	&V_{R_1}(t)= \tau_M\int_{-\tau_M}^{0} \int_{t+\theta}^t e^{-2\delta_0 (t-\tau)}\| \dot{X}_0(\tau) \|_{R_1}^2 \dd \tau, 
\end{align*}
where $V_{S_1}, V_{R_1}$  are introduced for compensating $v_{\tau}$. Differentiation of $V_{S_1}(t)$ and $V_{R_1}(t)$ yields
{\small \begin{subequations}\label{ly:Differentiation-V_R-V_s}
\begin{align}
    \dot{V}_{S_1} &+ 2 \delta_0 V_{S_1} \leqslant \| X_0 \|_{S_1}^2 - \varepsilon_M \|X_0(t)- v_{\tau}(t)- \theta_{\tau} (t) \|_{S_1}^2, \\
    \dot{V}_{R_1} &+ 2 \delta_0 V_{R_1} \leqslant \tau_M^2 \|\dot{X}_0\|_{R_1}^2 - \varepsilon_M \Psi_1^\top \sm{R_1 & G_1 \\ * & R_1 } \Psi_1, \label{ineq:park1}
\end{align}
\end{subequations}}%
 where $\Psi_1=\col \{v_\tau, \theta_\tau \}$ and the inequality in \eqref{ineq:park1} follows from Jensen's and Park's inequalities \cite{MR3237720}. 
 To compensate $\zeta(t)$, we have 
\begin{multline}\label{ly:ineq-Hala}
    -2\delta_1 \sup_{-\tau_M \leqslant \theta \leqslant 0} V_h(t+\theta) \leqslant -2\delta_1\| X_0(t)-v_{\tau}(t)\|_P^2 \\
    - 2 \delta_1 \|c\|_N^{-2} \zeta^2(t-\tau).
\end{multline}
Denote
 \begin{equation}\label{def:Halanay}
 	\mathcal{H}_d(t):= \dot{V}_h(t)+2 \delta_0 V_h(t) -2\delta_1 \sup_{-\tau_M \leqslant \theta \leqslant 0} V_h(t+\theta).
 \end{equation}
Let $ \eta(t)= \col \{X_0(t), \zeta(t-\tau),  v_{\tau}(t), \theta_{\tau}(t) \} $.
Taking the derivative of $V(t)$ as defined in \dref{ly:V} along \dref{eq:clo-OutputDelay} gives
\begin{equation}\label{ly:derivative-V}
	\begin{aligned}	
	&\dot{V}(t)+2\delta_0 V(t)= X_0^\top(t) (P F + F^T P+2 \delta_0 P) X_0(t)\\ 
	& +2 X_0^\top(t) P \Lscr \zeta(t-\tau)-2 X_0^\top (t) P \Lscr_2 v_{\tau}(t) \\
& + 2\sum_{n=N+1}^\infty (-\lambda_n+\delta_0)\, z_n^2(t)
+ 2\sum_{n=N+1}^\infty z_n(t)\, b_n\, \bar{K}\, X_0(t).
	\end{aligned}
\end{equation}
Let $ \eta_1(t)= \col \{X_0(t), \zeta(t-\tau),  v_{\tau}(t), \theta_{\tau}(t)\}$. 
Based on \eqref{ineq:1}, \eqref{ineq:2}, \dref{ly:Differentiation-V_R-V_s},  \dref{ly:ineq-Hala} and \dref{ly:derivative-V}, provided $M_\alpha<0$ and $\alpha>\frac{1}{2}$, we have that
\[\mathcal{H}_d(t) \leqslant \eta_1(t)^\top (\bar{H}+ \tau_M^2 \Phi_1^\top  R_1 \Phi_1)\, \eta_1(t)+ 2 \sum_{n=N+1}^\infty \rho_n z_n^2 ,\]
where $\rho_n=-\lambda_n+\delta_0+\frac{1}{2\alpha} \lambda_{n}.$  By the monotonicity of $\lambda_n$, $M_\alpha<0$ always holds for large $N$ and $\alpha>\frac{1}{2}$.

To show the feasibility of $\mathcal{H}_d(t)\leqslant 0,$ it suffices to show $\bar{H}+\tau_M^2 \Phi_1^\top  R_1 \Phi_1\leqslant 0$, since $\rho_n< 0$ for large $N$ and $\alpha>\frac{1}{2}$. For the asymptotic feasibility of $\bar{H}+\tau_M^2 \Phi_1^\top  R_1 \Phi_1\leqslant 0$ with large  $N$  and small $\tau_M$, let $S_1=G_1=0$ and $\tau_M$ tend to $0^+$. It is sufficient to show $ \Omega=\left[
\begin{matrix}
\Omega_1   & \Lambda_3  \\
* & \Omega_2
\end{matrix}
\right]\leqslant 0$, where  
\[ \begin{aligned}
\Omega_1 &= \begin{pmatrix} M_{11} & M_{12} \\ * & H_{22} \end{pmatrix}, \;
\Lambda_3 = \begin{pmatrix} -P\mathcal{L}_2 + 2\delta_1 P & 0 \\ 0 & 0 \end{pmatrix}, \\
\Omega_2 &= \text{diag}\left\{-\delta_1 P - R_1, \ -\varepsilon_M R_1 \right\}.
\end{aligned}\]
 Note the norms of $P$, $\Lscr$ and $\Lscr_2$ are independent of $N$.
 Taking $R_1=N^{2.5} I,$ using Schur complement, by  arguments similar to the proof of Theorem \ref{thm:main}, it can be obtained that $\Omega\leqslant 0$ for large enough  $N$.
 \end{myproof}
 

\bibliographystyle{plain}        
\bibliography{References}           
\appendix

\end{document}